\documentclass[psamsfonts,11pt]{amsart}
\usepackage{mathpazo}

\usepackage{overpic}
\usepackage{amsmath,amsthm,amssymb,amscd,amsfonts,amsbsy}
\usepackage{color}
\usepackage{hyperref,url}
\usepackage{float}
\usepackage{hyperref}
\usepackage{enumerate}
\usepackage{enumitem}   
\usepackage{mathrsfs}  
\usepackage{mathrsfs}  
\usepackage{amsthm}
\usepackage{amsmath}
\usepackage{amsfonts}
\usepackage{amssymb}
\usepackage{amsthm}
\usepackage{amsmath}
\usepackage{amsfonts}
\usepackage{amssymb}
\usepackage{mathtools}
\usepackage{mathrsfs}
\usepackage{comment}
\usepackage[top=3cm,bottom=2cm,left=2.5cm,right=2.5cm]{geometry}
\usepackage[mathlines]{lineno}

\usepackage{amsmath, amssymb, graphics, setspace}
\usepackage[utf8]{inputenc}
\usepackage{listings,xcolor}

\usepackage{tikz}

\newcommand\encircle[1]{%
  \tikz[baseline=(X.base)] 
    \node (X) [draw, shape=circle, inner sep=0] {\strut #1};}

\newcommand{\mathsym}[1]{{}}
\newcommand{\unicode}[1]{{}}

\newcommand{\R}{\ensuremath{\mathbb{R}}}

\newcommand{\CO}{\ensuremath{\mathcal{O}}}

\newcommand{\f}{\varphi}

\newcommand{\e}{\varepsilon}

\newcommand{\sgn}{\mathrm{sign}}

\usepackage[toc,page]{appendix}
\usepackage{tikz-cd}

\newtheorem {proposition}{Proposition}

\newtheorem {lemma} {Lemma}
\newtheorem {example} {Example}
\newtheorem {remark}{Remark}

\newtheorem {mtheorem} {Theorem}

\def\R{\mathbb R}

\title[Cyclicity of monodromic tangential singularities]
{Unveiling the Cyclicity of Monodromic Tangential Singularities: Insights Beyond the Pseudo-Hopf Bifurcation}
\author[D. D. Novaes and L. A. Silva]
{Douglas D. Novaes and Leandro A. Silva}

\address{Departamento de Matem\'{a}tica - Instituto de Matem\'{a}tica, Estat\'{i}stica e Computa\c{c}\~{a}o Cient\'{i}fica (IMECC) - Universidade
Estadual de Campinas (UNICAMP), \ Rua S\'{e}rgio Buarque de Holanda, 651, Cidade Universit\'{a}ria Zeferino Vaz, 13083-859, Campinas, SP,
Brazil}
\email{ddnovaes@unicamp.br}
\email{lasilva@ime.unicamp.br}

\allowdisplaybreaks
\begin{document}



\maketitle
\begin{abstract}
The cyclicity problem, crucial in analyzing planar vector fields, consists in estimating the number of limit cycles emanating from monodromic singularities.  Traditionally, this estimation relies on Lyapunov coefficients. However, in nonsmooth systems, besides the limit cycles bifurcating by varying the Lyapunov coefficients, monodromic singularities on the switching curve can always be split apart yielding, under suitable conditions, a sliding region and an additional limit cycle surrounding it. This bifurcation phenomenon, known as pseudo-Hopf bifurcation, has enhanced lower-bound cyclicity estimations for monodromic singularities in Filippov systems. In this study, we push beyond the pseudo-Hopf bifurcation, demonstrating that the destruction of $(2k,2k)$-monodromic tangential singularities yields at least $k$ limit cycles surrounding sliding segments. This new bifurcation phenomenon expands our understanding of limit cycle bifurcations in nonsmooth systems and, in addition to the theoretical significance, has practical relevance in various applied models involving switches and abrupt processes.
\end{abstract}

\section{Introduction}
The classical and important topic in the qualitative theory of planar vector fields known as the {\it cyclicity problem} addresses the estimation of the number of limit cycles bifurcating from a monodromic singularity $p$ (center or focus) when perturbed within a given family $\mathcal{F}$ of vector fields. Such a number of limit cycles is called {\it cyclicity} of the monodromic singularity $p$ in $\mathcal{F}$. This problem has garnered extensive attention, both in the realm of planar polynomial vector fields (see, for instance, \cite{PT19,romanovski2009center}) and in the context of piecewise smooth vector fields (see, for instance, \cite{gassulcoll,cruz19,GasTor03,GouTor20}). The study of limit cycles emanating from monodromic singularities often employs Lyapunov coefficients, denoted as $V_i$'s, which are derived from the coefficients in the power series expansion of the first-return map $\pi(x)$ around the singularity. This approach has been explored in both smooth systems \cite{romanovski2009center} and nonsmooth systems \cite{gassulcoll,GasTor03,novaessilva2020}, particularly when $\pi$ exhibits analytic behavior in a neighborhood of $x=0$, where $\pi(x)-x=\sum_{i\geq 1} V_i x^i$ for $x$ close to $0$.

Differential equations featuring discontinuities constitute a crucial class of dynamical systems with wide-ranging applications in applied science, particularly in scenarios involving switches and abrupt processes, (see, for instance \cite{andronov1966theory,CRM,BBCK,Mike18}). Filippov's book \cite{Filippov88} introduced a rigorous mathematical framework for nonsmooth differential equations, now commonly referred to as Filippov systems. The notion of singularity in Filippov systems encompasses traditional singular points found in smooth differential systems, as well as new types of points arising from nonsmoothness, namely, pseudo-equilibria and tangency points (see, for example, \cite{guardia2011generic}). These additional singularities can give rise to local monodromic behavior, thereby expanding the range of singularities for which the cyclicity problem is well-posed.

As mentioned before, the cyclicity problem has been investigated in Filippov systems. For example, Filippov computed some Lyapunov coefficients for a monodromic singularity of fold-fold type (also called sewed focus) in Chapter 4 of his book \cite{Filippov88}.  Coll et al., in \cite{CGP95}, obtained the first seven Lyapunov constants for monodromic singularities of focus-focus type in discontinuous Li\'{e}nard differential equations; then,  in \cite{coll1999}, they  derived general expressions for the Lyapunov constants for monodromic singularities of focus-focus type in some families of discontinuous Li\'{e}nard differential equations. They also addressed the cyclicity problem for monodromic singularities of focus-focus type, fold-fold type, and focus-fold type, explicitly computing the first three Lyapunov coefficients for these types of singularities \cite{gassulcoll}. Gasull and Torregrosa also studied the cyclicity problem for monodromic singularities of focus-focus type in several classes of piecewise smooth systems \cite{GasTor03}. More recently, in \cite{Glendinning2023}, Glendinning et. al studied more general sewed focuses and demonstrated that, in the non-analytic case, this singularity can have a sequence of limit cycles accumulating on it.

In this paper, we focus on $(2k,2k)$-monodromic tangential singularities, which arise when $2k$-multiplicity tangencies between the switching curve and the vector fields coalesce and exhibit locally monodromic behavior around the singularity. Recently, in \cite{novaessilva2020}, a recursive formula was introduced for computing all the Lyapunov coefficients associated with these singularities. For further discussion on the isochronicity and criticality problems related to this type of singularity, see \cite{NS22}. We note that for $k=1$, this class of singularities corresponds to the fold-fold type.

The discontinuities presented in Filippov systems can induce a richer variety of bifurcation phenomena, some of them unique to nonsmooth systems, often called by {\it discontinuity-induced bifurcations} (DIBs) (see \cite{bernardo} for a discussion on DIBs). Regarding the cyclicity problem, in the nonsmooth context, besides the limit cycles bifurcating by varying the Lyapunov coefficients, monodromic singularities lying on the switching curve can always be split apart generating, under suitable conditions, a sliding region and an extra limit cycle surrounding it (see Figure \ref{figpseudo}). This DIB is often referred to as {\it pseudo-Hopf bifurcation} and it has been firstly reported by Filippov in his book \cite{Filippov88}. This bifurcation phenomenon has been used to investigate the cyclicity of monodromic singularities in Filippov systems, which allows to increase in the obtained lower bounds for the cyclicity at least by one  (see, for instance, \cite{cruz19,GouTor20,novaessilva2020}).

In this paper, our objective is to extend and improve the analysis of pseudo-Hopf bifurcations to monodromic tangential singularities. We aim to demonstrate that the number of limit cycles generated by destroying such singularities increases with the degeneracy. Specifically, we will establish that the destruction of $(2k,2k)$-monodromic tangential singularities leads to the emergence of at least $k$ small amplitude limit cycles surrounding sliding segments (see Figure \ref{fig:main}). The formal definition of monodromic tangential singularities is provided in Section \ref{sec:mts} below. Additionally, the formal statement of the pseudo-Hopf bifurcation for such singularities can be found in Section \ref{sec:phb} (see also the Appendix). Our main result is presented in Section \ref{sec:main}.

It is worth mentioning that, in addition to the theoretical significance that underpins the study of local bifurcation of limit cycles in nonsmooth dynamical systems (here, we must refer to \cite{SIMPSON20182439}, where $20$ geometric mechanisms by which limit cycles are created locally in planar piecewise smooth vector fields are discussed), the practical relevance of this research is underscored by its connection to observed phenomena in various applied models. For instance, the local bifurcation of limit cycles is a phenomenon exhibited in several applied piecewise smooth models, including neuron models \cite{HB,NS,mckean}, the Gause predator-prey model \cite{krivan}, and models of car braking systems \cite{zou}. By drawing parallels with these applied models, the study of the new bifurcation phenomenon presented in this paper gains further importance and relevance.

\subsection{Monodromic tangential singularities}\label{sec:mts}
Consider the following piecewise smooth vector field:
\begin{equation} \label{sistemainicial}
Z(x,y)= \begin{cases}
Z^+(x,y)= \left(\begin{array}{c} X^+(x,y)\\ Y^+(x,y)\end{array}\right), & y > 0,\vspace{0.3cm} \\
Z^-(x,y)=\left(\begin{array}{c} X^-(x,y)\\ Y^-(x,y)\end{array}\right),& y < 0.
\end{cases}
\end{equation}
Throughout this paper, the Filippov's convention \cite{Filippov88} will be assumed for trajectories of \eqref{sistemainicial}.
 
Recall that the origin $(0,0)$ is a {\it contact of multiplicity $n$} (or {\it order} $n-1$) between $Z^{+}$ and $\Sigma=\{(x,y)\in\R^2:\,y=0\}$ if $x=0$ is a root of order $n-1$ of $x\mapsto Y^{+}(x,0).$ In other words,
\[
Y^{+}(0,0)=0,\,\dfrac{\partial^i Y^{+}}{\partial x^i}(0,0)=0\, \text{for}\, i\in\{1,\ldots,n-2\},\,\text{and}\, \dfrac{\partial^{n-1} Y^{+}}{\partial x^{n-1}}(0,0)\neq0.
\] 
The same concept holds for $Z^{-}$. In addition, an even multiplicity contact, let us say $n=2k,$ is called {\it invisible} for $Z^+$ and $Z^-$ when 
\[
\dfrac{\partial^{2k-1} Y^{+}}{\partial x^{2k-1}}(0,0)<0\,\text{and}\, \dfrac{\partial^{2k-1} Y^{-}}{\partial x^{2k-1}}(0,0)>0,
\]
respectively. Otherwise, it is called {\it visible}.

Accordingly, we say that \eqref{sistemainicial} has a {\it $(2k^+,2k^-)$-monodromic tangential singularity} at the origin, provided that:

\bigskip

\noindent{\bf C1.} $X^{\pm}(0,0)\neq0,$ $Y^{\pm}(0,0)=0,$ $\dfrac{\partial^i Y^{\pm}}{\partial x^i}(0,0)=0$  for  $i\in\{1,\ldots,2k^{\pm}-2\},$ and  $\dfrac{\partial^{2k^{\pm}-1} Y^{\pm}}{\partial x^{2k^{\pm}-1}}(0,0)\neq0$;

\medskip

\noindent{\bf C2.} $X^+(0,0)\dfrac{\partial^{2k^{+}-1} Y^{+}}{\partial x^{2k^{+}-1}}(0,0)<0$ and $X^-(0,0)\dfrac{\partial^{2k^{-}-1} Y^{-}}{\partial x^{2k^{-}-1}}(0,0)>0;$

\bigskip

\noindent{\bf C3.} $X^+(0,0) X^-(0,0)<0.$

\bigskip

\noindent 
Briefly speaking, {\bf C1} and {\bf C2} ensure that the origin is an invisible $2k^{+}$-multiplicity (resp. $2k^{-}$-multiplicity) contact between $\Sigma$ and the vector field $Z^{+}$ (resp.  $Z^{-}$) and {\bf C3} ensures that a first-return map is well defined on $\Sigma$ around the origin (see Figure \ref{fig:2kmts}).

	\begin{figure}[H]
     	\centering 
     	\begin{overpic}[scale=0.6]{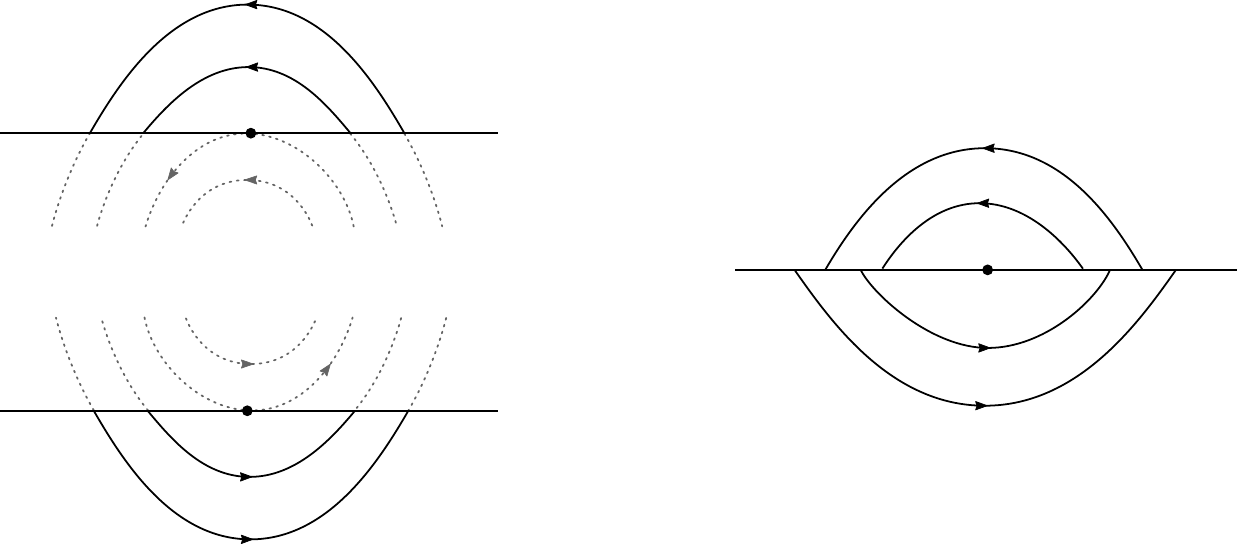}
     		\put(41,32.5){$\Sigma$}
     		\put(21,30){\small $0$}
     		\put(41,10){$\Sigma$}
     		\put(21,8){\small $0$}
     		\put(101,22){$\Sigma$}
     		\put(80,19){\small $0$}
     		\put(32,42){$Z^+$}
     		\put(32,0){$Z^-$}
     		\put(92,30){$Z^+$}
     		\put(92,12){$Z^-$}
     		\put(50,21){$\longrightarrow$}
     	\end{overpic}
     	\vspace{0.5cm}
     	\caption{Invisible $2k$-multiplicity contact between $\Sigma$ and the vector fields $Z^+$ and $Z^-$. The concatenation of the orbits of $Z^+$ and $Z^-$ through $\Sigma$ allows the definition of  a first-return map. The origin is a $(2k,2k)$-monodromic tangential singularity of $Z$.}
     	\label{fig:2kmts}
     \end{figure}
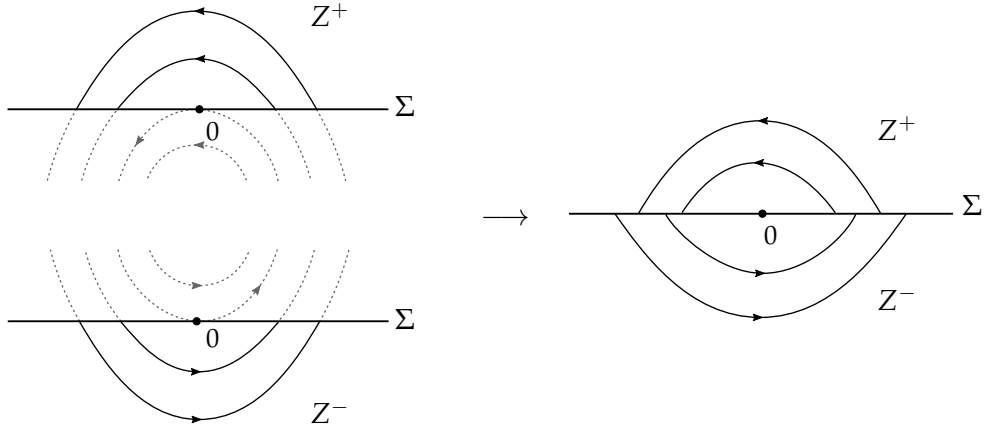

In \cite{novaessilva2020}, the $(2k^+,2k^-)$-monodromic tangential singularities were deeply investigated. In particular, a recursive formula for all the Lyapunov coefficients was provided. When dealing with monodromic tangential singularities, instead of working with the first-return map, it is more convinient to consider the {\it displacement function}
\begin{equation}\label{def:disp}
\Delta(x)=(\varphi^+(x)-\varphi^-(x))\delta,
\end{equation}
where  
\[
\delta=\sgn(X^+(0,0))=-\sgn(X^-(0,0)),
\]
and $\varphi^{+}$ and $\varphi^{-}$ are the half-return maps defined on $\Sigma$ around $0$ by means of flows of the $Z^{+}$ and $Z^{-}$ restricted, respectively, to $\Sigma^{+}=\{(x,y):\, y\geq 0\}$ and $\Sigma^{-}=\{(x,y):\, y\leq 0\}$ (see Figure \ref{involucoes}). Such maps are involutions satisfying $\varphi^{+}(0)=\varphi^-(0)=0$, which are analytic provided that $Z^+$ and $Z^-$ are analytic (see \cite[Theorem A]{novaessilva2020}). In this case, the Lyapunov coefficients are determined by the power series expansion of $\Delta$ around $x=0$ as follows:
\begin{equation}\label{expansion}
\Delta(x)= \sum_{i=1}^{\infty}V_i x^i.
\end{equation}
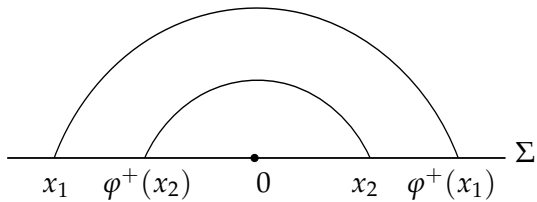
\begin{figure}[H]
	\centering 
	\begin{overpic}[scale=0.6]{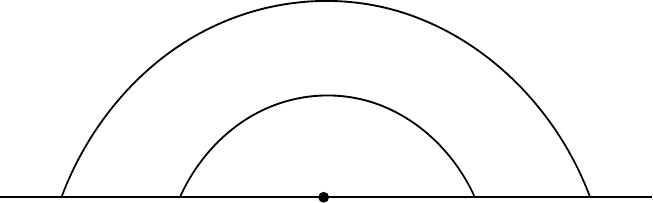}
		\put(7,-6){$x_1$}	
		\put(19,-6){$\varphi^+(x_2)$}	
		\put(50,-6){$0$}
		\put(80,-6){$\varphi^+(x_1)$}
		\put(69,-6){$x_2$}
		\put(102,0){$\Sigma$}
	\end{overpic}
\vspace{0.5cm}
	\caption{Illustration of half-return map $\varphi^+.$}
	\label{involucoes}
\end{figure}

A $(2k^+,2k^-)$-monodromic tangential singularity can be a center, provided that the displacement function is identically zero, or a focus. In the last case, the singularity is attracting or repelling provided that the displacement function is increasing or decreasing close to the origin. By the definition of the Lyapunov coefficients, $V_i$'s, they can reveal if the singularity is a focus and also provide its stability. Indeed, the existence of a  non-vanishing Lyapunov coefficient implies that the singularity is a focus and the sign of the first non-vanishing Lyapunov coefficient provides that the focus is attracting or repelling if it is negative or positive, respectively. In \cite[Theorem B]{novaessilva2020}, it was established that the index of the first non-vanishing Lyapunov coefficient of a monodromic tangential singularity is always even. Additionally, \cite[Theorem C]{novaessilva2020} provided a recursive formula to compute all Lyapunov coefficients. 

\subsection{Canonical form and the second Lyapunov coefficient}\label{sec:cf} The study in \cite{novaessilva2020} was based on a canonical form for Filippov systems around a $(2k^+, 2k^-)$-monodromic tangential singularity, introduced in \cite{ANDRADE2023101397}, which we will discuss in the sequel.

For a Filippov system \eqref{sistemainicial} with a $(2k^{+}, 2k^{-})$-monodromic tangential singularity at the origin, since $X^{\pm}(0,0)\neq0,$ there exists a small neighborhood $U$ of the origin such that $X^{\pm}(x,y) \neq0$ for all $(x,y) \in U.$  Taking into account that $|X^{\pm}(x,y)| = \pm \delta X^{\pm}(x,y)$ for every $(x,y)\in U,$ a time rescaling can be performed in order to transform the Filippov system \eqref{sistemainicial} restricted to $U$ into
 \begin{equation}\label{cf}
(\dot{x},\dot{y}) = \tilde{Z}(x,y)=\begin{cases}
(\delta, \eta^+(x,y)), &y > 0, \\
(-\delta, \eta^-(x,y)), &y < 0,
\end{cases}
\end{equation}
 where 
\begin{equation} \label{eq:eta}
\eta^{\pm}(x,y) :=\dfrac{Y^{\pm}(x,y)}{|X^{\pm}(x,y)|}= \pm\delta\dfrac{Y^{\pm}(x,y)}{X^{\pm}(x,y)}. 
\end{equation}
In addition, working out the condition {\bf C1}, one can see that $\eta^{\pm}(x,0) = a^{\pm}x^{2k^{\pm}-1} + x^{2k^{\pm}}f^{\pm}(x)$ and, therefore,
\begin{equation}\label{etapm}
\eta^{\pm}(x,y) = a^{\pm}x^{2k^{\pm}-1} + x^{2k^{\pm}}f^{\pm}(x) + yg^{\pm}(x,y).
\end{equation} 
The coefficients $a^{\pm}$ and the functions $f^{\pm}(x)$ and $g^{\pm}(x,y)$ can be computed from \eqref{etapm} as follows
\begin{equation} \label{value:a}
a^{\pm}=\dfrac{1}{(2k^{\pm}-1)!|X^{\pm}(0,0)|}\dfrac{\partial^{2k^{\pm}-1}Y^{\pm}}{\partial x^{2k^{\pm}-1}}(0,0),
\end{equation}
\begin{equation}\label{auxfunc}
\begin{array}{l}
f^{\pm}(x)=\dfrac{\pm \delta Y^{\pm}(x,0)-a^{\pm}x^{2k^{\pm}-1} X^{\pm}(x,0)}{x^{2k^{\pm}} X^{\pm}(x,0)},\,\,\,\text{and}\vspace{0.2cm}\\
g^{\pm}(x,y)=\dfrac{\pm X^{\pm}(x,0)Y^{\pm}(x,y) \mp X^{\pm}(x,y)Y^{\pm}(x,0)}{y \delta X^{\pm}(x,y)X^{\pm}(x,0)}.
\end{array}
\end{equation}
Notice that the functions $f^{\pm}(x)$ and $g^{\pm}(x,y)$ are as smooth as $Z^{\pm}$ in a neighborhood of $x=0$ and $(x,y)=(0,0)$, respectively. 

Then, define $f^{\pm}_0=f^{\pm}(0)$ and  $g^{\pm}_{0,0}=g^{\pm}(0,0)$. From \cite[Corollary 1]{novaessilva2020} (see also \cite{gassulcoll}, in the case $k^+=k^-=1$), we have that the second Lyapunov coefficient is given by
\begin{equation}\label{eq:coefV2NP}
V_2=\delta(\alpha_2^+-\alpha_2^-), \,\,\, \text{where}\,\,\,\alpha_2^{\pm}=\dfrac{-2f^{\pm}_0\pm2 \delta a^{\pm} g^{\pm}_{0,0}}{a^{\pm}(2k^{\pm}+1)}.
\end{equation}
Indeed, denote by $\phi^{\pm}(t,x_0) = (x^{\pm}(t,x_0), y^{\pm}(t,x_0))$ the solutions of $(\dot{x},\dot{y}) = (\pm \delta, \eta^{\pm}(x,y))$
with initial condition $\phi^{\pm}(0,x_0) = (x^{\pm}(0,x_0), y^{\pm}(0,x_0))=(x_0,0)\in \Sigma.$ Notice that $x^{\pm}(t,x_0)=x_0 \pm \delta t,$ so that $x^{\pm}(t,x_0) = 0$ if, and only if,  $t= \mp \delta x_0.$ Thus, by defining
\begin{equation}\label{muproof}
\mu^{\pm}(x_0)=y^{\pm}(\mp \delta x_0,x_0),
\end{equation} 
we see that $\mu^{\pm}(x_0)=\mu^{\pm}(\varphi^{\pm}(x_0)),$ for every $x_0$ in a neighborhood of the origin (see Figure \ref{mapadobradobra}). Developing the relationship \eqref{muproof} in power series around $x_0=0$, we can see that
\[
\varphi^{\pm}(x)=-x+\alpha_2^{\pm} x^2+\mathcal{O}(x^3),
\] 
where $\alpha_2^{\pm}$ is given as in \eqref{eq:coefV2NP}. Thus, taking \eqref{def:disp} into account, we get $V_2$ as given by \eqref{eq:coefV2NP}.
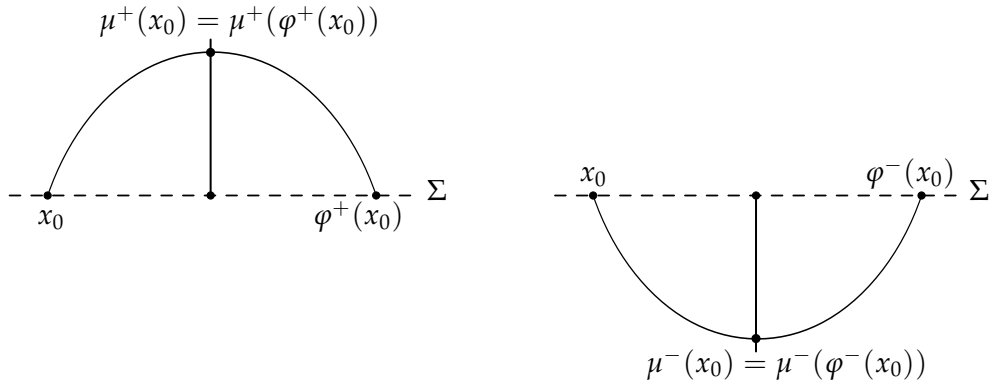
\begin{figure}[H]
	\centering
	\vspace{0.5cm}
	\begin{overpic}[scale=0.6]{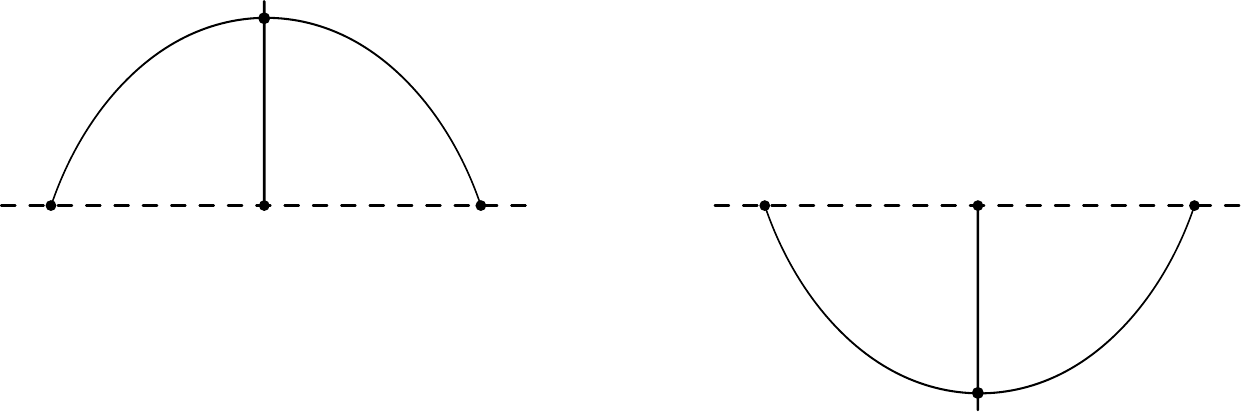}
		\put(44,16){$\Sigma$}
		\put(101,16){$\Sigma$}
		\put(9.5,34){$\mu^+(x_0)=\mu^+(\varphi^+(x_0))$}
		\put(67,-2){$\mu^-(x_0)=\mu^-(\varphi^-(x_0))$}
		\put(3,13.5){$x_0$}
		\put(60,18){$x_0$}
		\put(32,13.5){$\varphi^+(x_0)$}
		\put(90,18){$\varphi^-(x_0)$}
	\end{overpic}
	\vspace{0.5cm}
	\caption{Representation of the maps $\mu^{+}$ and $\mu^{-}$ satisfying the relationship \eqref{muproof}.}
	\label{mapadobradobra}
\end{figure}

\subsection{The Pseudo-Hopf Bifurcation}\label{sec:phb}
The Pseudo-Hopf bifurcation has been firstly mentioned by Filippov in his book \cite{Filippov88} (see item b of page 241). In what follows we present the formal statement of the pseudo-Hopf bifurcation in the case of the so-called invisible two-fold singularity which, in our terminology,  corresponds to a $(2,2)$-monodromic tangential singularity.

\begin{proposition}\label{prop:phd1}
Assume that the Filippov system $Z$, given by \eqref{sistemainicial}, has a $(2,2)$-monodromic tangential singularity at the origin  with non-vanishing second Lyapunov coefficient $V_2$. 
Consider the 1-parameter family of Filippov systems
\begin{equation}\label{Zb}
Z_{b}(x,y)=\begin{cases}
Z^+(x-b,y),&y>0,\\
Z^-
(x,y),&y<0.
\end{cases}
\end{equation}
Then, given a neighborhood $U\subset \R^2$ of $(0,0),$ there exists a neighborhood $I\subset\R$ of $0$ such that the following statements hold for every $b\in I$:
\begin{itemize}
\item If $\sgn(b)=-\sgn(\delta V_2)$, then $Z_{b}$ has a hyperbolic limit cycle in $U$ surrounding a sliding segment  (see Figure \ref{figpseudo}). In addition, the hyperbolic limit cycle is stable (resp. unstable) provided that $V_{2}<0$ (resp. $V_{2}>0$).

\item If $\sgn(b)=\sgn(\delta V_2)$, then $Z_{b}$ does not have limit cycles in $U$.
\end{itemize}
\end{proposition}
In Appendix A, we provide an original and more general version of Proposition \ref{prop:phd1} that establishes the bifurcation of a hyperbolic limit cycle from a $(2k^+,2k^-)$-monodromic tangential singularities provided that some Lyapunov coefficient does not vanish.

\begin{figure}[H]
	\bigskip
	\centering
	\begin{overpic}[scale=0.6]{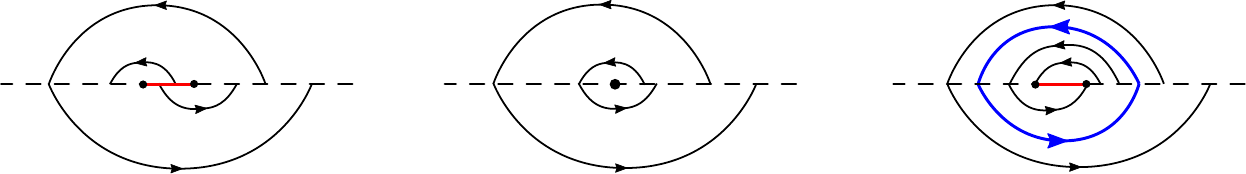}
		\put(10,-3){$b<0$}
		\put(47,-3){$b=0$}
		\put(84,-3){$b>0$}
		\put(29,7){$\Sigma$}
		\put(64.5,7){$\Sigma$}
		\put(100.5,7){$\Sigma$}
	\end{overpic}
	\vspace{0.8cm}
	
	\caption{ In this figure, the origin is a repelling two-fold singularity of $ Z_0$ which undergoes a pseudo-Hopf bifurcation as $b$ varies. For $b\neq0$ the two-fold singularity is split into two regular-fold singularities and between them a sliding segment is created, which is repelling for $b<0$ and attracting for $b>0$. In the last case, an attracting hyperbolic limit cycle surrounding the attracting sliding segment is created. Continuous and dashed segments on $\Sigma$ represent sliding and crossing regions, respectively.}
	\label{figpseudo}
\end{figure}
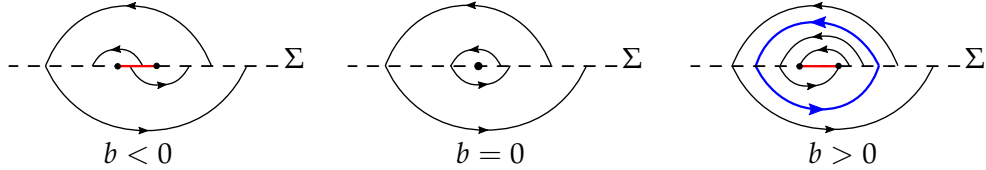

\subsection{Main Result}\label{sec:main} Our main result improves Proposition \ref{prop:phd1} by showing that at least $k$ limit cycles bifurcate  from $(2k,2k)$-monodromic tangential singularity when it is destroyed by considering a perturbation that adds some ``missing'' lower degree terms to $Z$. This implies that this kind of singularity has cyclicity at least $k$. Thus, any number of limit cycles obtained without destroying the monodromic singularity (for instance, by varying the Lyapunov coefficients) can be increased at least by $k$.

In what follows the concept of {\it norm of a polynomial} is the usual one, that is, the norm of the parameter-vector that defines it.

\begin{mtheorem}\label{thm:unfold} Let $k$ be a positive integer. Assume that the Filippov system $Z$, given by \eqref{sistemainicial}, has a $(2k,2k)$-monodromic tangential singularity at the origin with non-vanishing second Lyapunov coefficient $V_2$. Then, given $\lambda>0$ and a neighborhood $U\subset \R^2$ of $(0,0)$,  there exist polynomials $P^+$ and $P^-$, with degree $2k-2$ and norm less than $\lambda$, and a neighborhood $I\subset \R$ of $0$ such that, for every $b\in I$ satisfying $\sgn(b)=-\sgn(\delta V_2)$, the  Filippov system
\begin{equation} \label{teo:sistemainicialperturbado}
\widetilde Z_b(x,y)=  \begin{cases}
 \widetilde Z^+(x-b,y), & y > 0,\vspace{0.3cm} \\
 \widetilde Z^-(x,y),& y < 0,
\end{cases}\,\,\,\text{with}\,\,\, \widetilde Z^{\pm}(x,y)=\left(\begin{array}{c} X^{\pm}(x,y)\\ Y^{\pm}(x,y)+X^{\pm}(x,y)P^{\pm}(x)\end{array}\right),
\end{equation}
has $k$ hyperbolic limit cycles inside $U$ such that each one of these limit cycles surrounds a single sliding segment (see Figure \ref{fig:main}).  In addition, the hyperbolic limit cycles are stable (resp. unstable) provided that $V_{2}<0$ (resp. $V_{2}>0$).
\end{mtheorem}

\begin{remark}
In Theorem \ref{thm:unfold}, the perturbation terms, $\widetilde P^{+}(x,y):=X^{+}(x,y)P^{+}(x)$ and $\widetilde P^{-}(x,y):=X^{-}(x,y)P^{-}(x)$, only add some ``missing'' lower degree terms to the canonical form  \eqref{cf} of the Filippov system $Z$. Indeed, the canonical form of the perturbed Filippov system \eqref{teo:sistemainicialperturbado} for $b=0$, $\widetilde Z^{\pm}_0(x,y)$, writes
\begin{equation*} 
(\dot{x},\dot{y}) = \begin{cases}
(\delta, \eta^{+}(x,y)+ \delta P^{+}(x)), &y > 0, \\
(-\delta, \eta^{-}(x,y)- \delta P^{-}(x)), &y < 0.
\end{cases}
\end{equation*}
Notice that  $P^+(x)$ and $P^-(x)$ are polynomials of degree $2k-2$ while, from \eqref{etapm}, $\eta^{\pm}(x,0)=a^{\pm}x^{2k-1} + x^{2k}f^{\pm}(x)$.

The role of the perturbation terms, $\widetilde P^{+}(x,y)$ and $\widetilde P^{-}(x,y)$, consists in destroying the $(2k,2k)$-monodromic tangential singularity of $Z$ at the origin and unfold from it $2k-1$ two-fold singularities for $\widetilde Z_0(x,y)$. Among those singularities, $k$ of them are $(2,2)$-monodromic tangential singularities, from which limit cycles are created via pseudo-Hopf bifurcation for $\widetilde Z_b(x,y)$, with $b\in I$ satisfying $\sgn(b)=-\sgn(\delta V_2)$ (see Figure \ref{fig:main}).
\end{remark}

\begin{figure}[H]
	\bigskip
	\centering
	\begin{overpic}[scale=0.6]{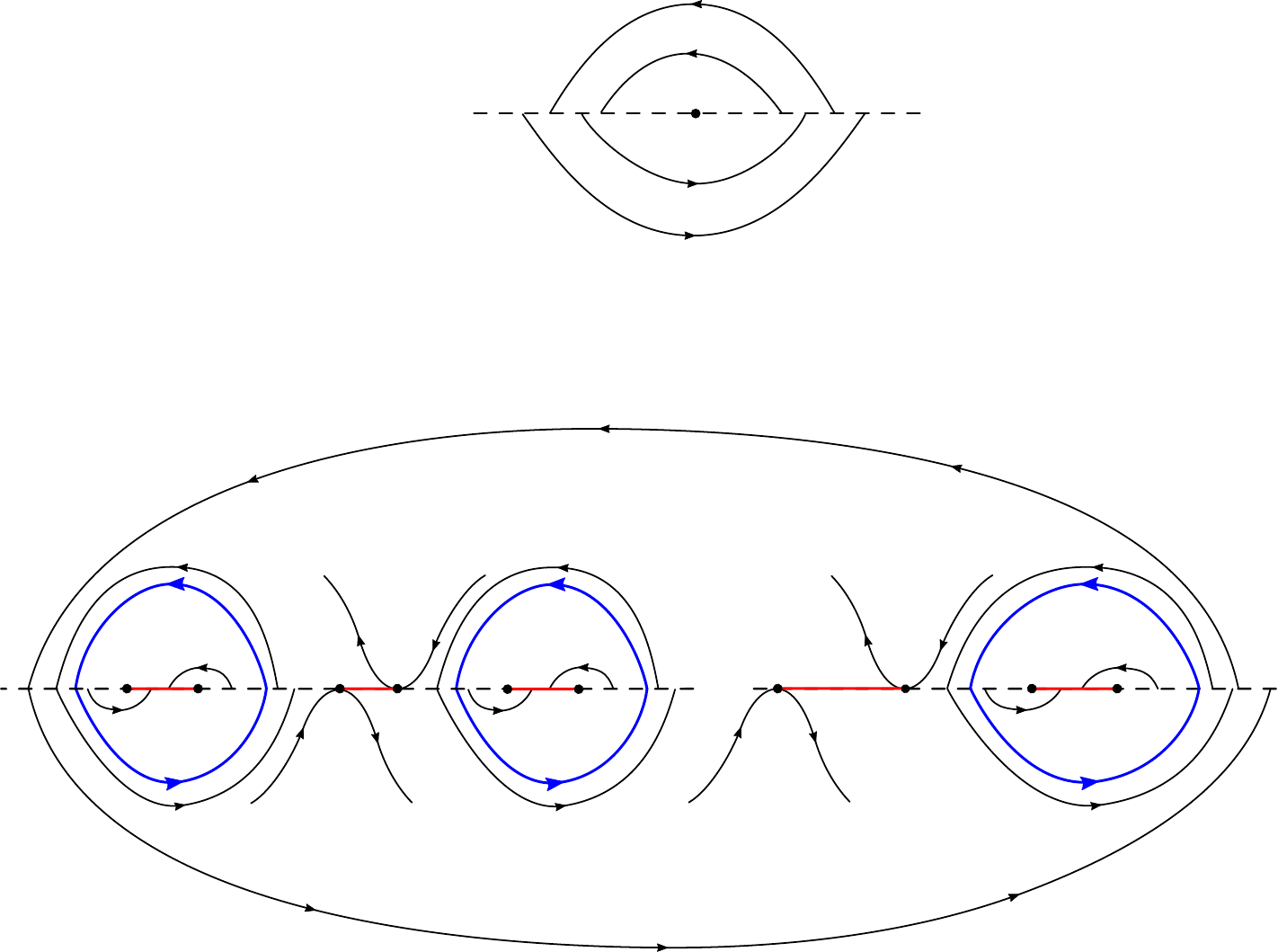}
		\footnotesize
		\put(52,63){$(0,0)$}
		\put(72.5,65){$\Sigma$}
		{\Large
			\put(53,48){$\downarrow$}}
		\put(53,52){$Z$}
		\put(53,43){$\widetilde{Z}_b$}
		\put(8,21.5){$\e a_1$}
		\put(38,21.5){$\e a_2$}
		\put(60,21.5){$\e a_{2k-3}$}
		\put(77,21.5){$\e a_{2k-2}$}
		\put(12,18){$\e a_1+b$}
		\put(26,21.5){$0$}
		\put(31,18){$b$}
		\put(42,18){$\e a_2+b$}
		\put(64,18){$\e a_{2k-3}+b$}
		\put(82,18){$\e a_{2k-2}+b$}
		\put(100.5,19.8){$\Sigma$}
		\tiny
		\put(54,19.9){/$\cdots$/}
	\end{overpic}
	\vspace{0.8cm}
	\caption{Illustration of $k$ limit cycles unfolding from a $(2k^+,2k^-)$-monodromic tangential singularity and surrounding $k$ sliding segments predicted by Theorem \ref{thm:unfold}. Continuous and dashed segments on $\Sigma$ represent sliding and crossing regions, respectively.}
	\label{fig:main}
\end{figure}
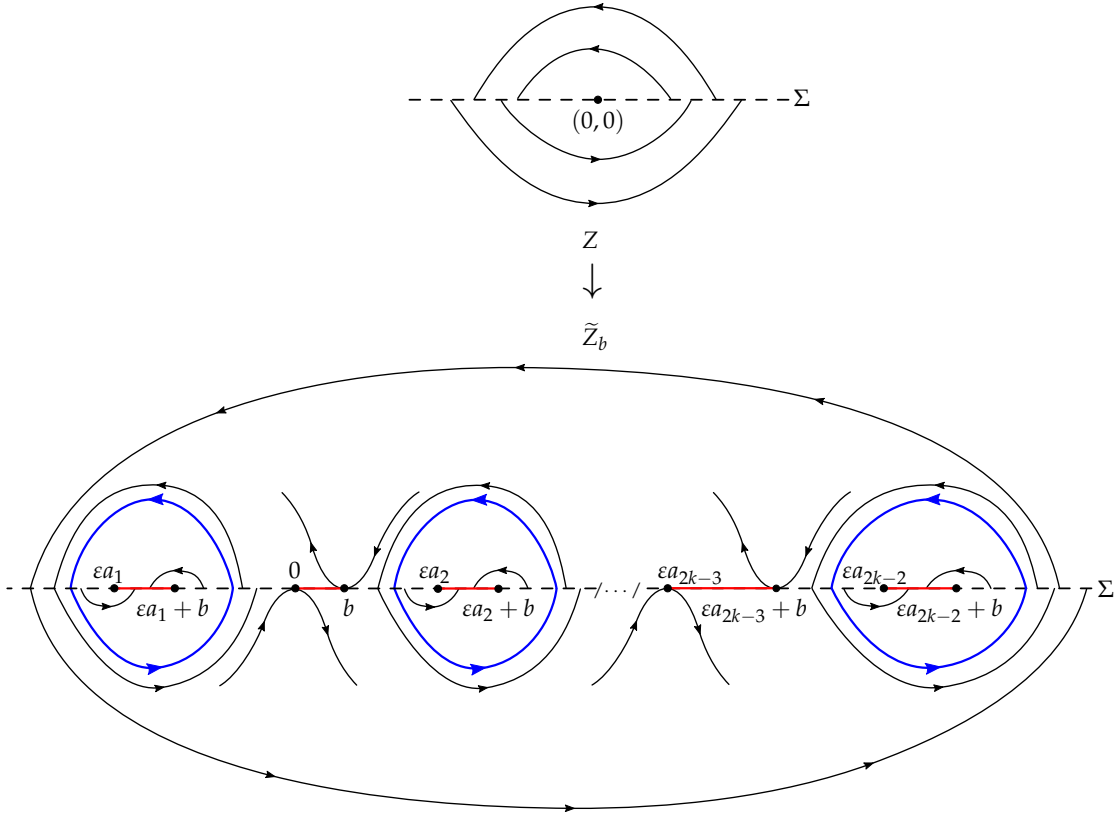

The proof of Theorem \ref{thm:unfold} is done in the next section. The idea is to construct polynomials $P_{\Lambda}^+$ and $P_{\Lambda}^-$ that unfold from the origin $2k-1$ two-fold singularities of which $k$ of them are $(2,2)$-monodromic tangential singularities with non-vanishing second Lyapunov coefficients. Then, the proof will follow by applying Proposition \ref{prop:phd1} which ensures the creation of a hyperbolic limit cycle surrounding a sliding segment from each one of the $(2,2)$-monodromic tangential singularities when they are destroyed.

\begin{example}
The following 1-parameter family of polynomial Filippov systems of degree $k$ has been studied in \cite[Example 2]{novaessilva2020}:
\begin{equation} \label{ex2}
Z_{\eta}(x,y)= \begin{cases}
\Big(1, x^{2k-1}(\eta\,x-1)+y \Big),&  y > 0, \\
\Big(-1,x^{2k-1}(x-1) \Big),& y < 0,
\end{cases}
\end{equation}
where $k$ is a positive integer and $\eta\in\R$. The origin is a $(2k,2k)$-monodromic tangential singularity for every $\eta\in\R.$ As an application of Theorem \cite[Theorem D]{novaessilva2020}, it was proved that  for  $\eta>0$ sufficiently small, the Filippov system \eqref{ex2} has an asymptotically stable hyperbolic limit cycle which converges to the origin as $\eta$ goes to zero. Moreover,
\[
V_2=\dfrac{2\eta}{1+2k}.
\]
Now, let $\eta>0$ be fixed such that \eqref{ex2} has an asymptotically stable hyperbolic limit cycle.  Applying Theorem \ref{thm:unfold}, we conclude that the Filippov system \eqref{ex2} can be perturbed inside the space of  polynomial Filippov systems of degree $2k-1$ in such a way that $k$ extra unstable limit cycles emerge from the origin. More specifically, given $\lambda>0$ and a neighborhood $U\subset \R^2$ of $(0,0)$,  there exist polynomials $P^+$ and $P^-$, with degree $2k-2$ and norm less than $\lambda$, and a neighborhood $I\subset \R$ of $0$ such that, for every $b\in I$ satisfying $\sgn(b)=-\sgn(\eta)$, i.e. $b<0$, the  Filippov system 
\begin{equation}\label{eq:example}
\widetilde Z(x,y)=\begin{cases}
\Big(1, (x-b)^{2k-1}(\eta\,(x-b)-1)+y+P^+(x-b) \Big),&  y > 0, \\
\Big(-1,x^{2k-1}(x-1) -P^-(x)\Big),& y < 0,
\end{cases}
\end{equation}
has an asymptotically stable limit cycle and $k$ unstable hyperbolic limit cycles. These $k$ limit cycles are inside $U$. In Figure \ref{fig:simula}, we present numerical simulations performed in Mathematica for $k=2$, $\eta=1/2$, $b=-1/400$, $P^+(x)=-x(x-2)/32$, and $P^-(x)=x(x-1)/16$.

	\begin{figure}[H]
     	\centering 
     	\begin{overpic}[scale=0.3]{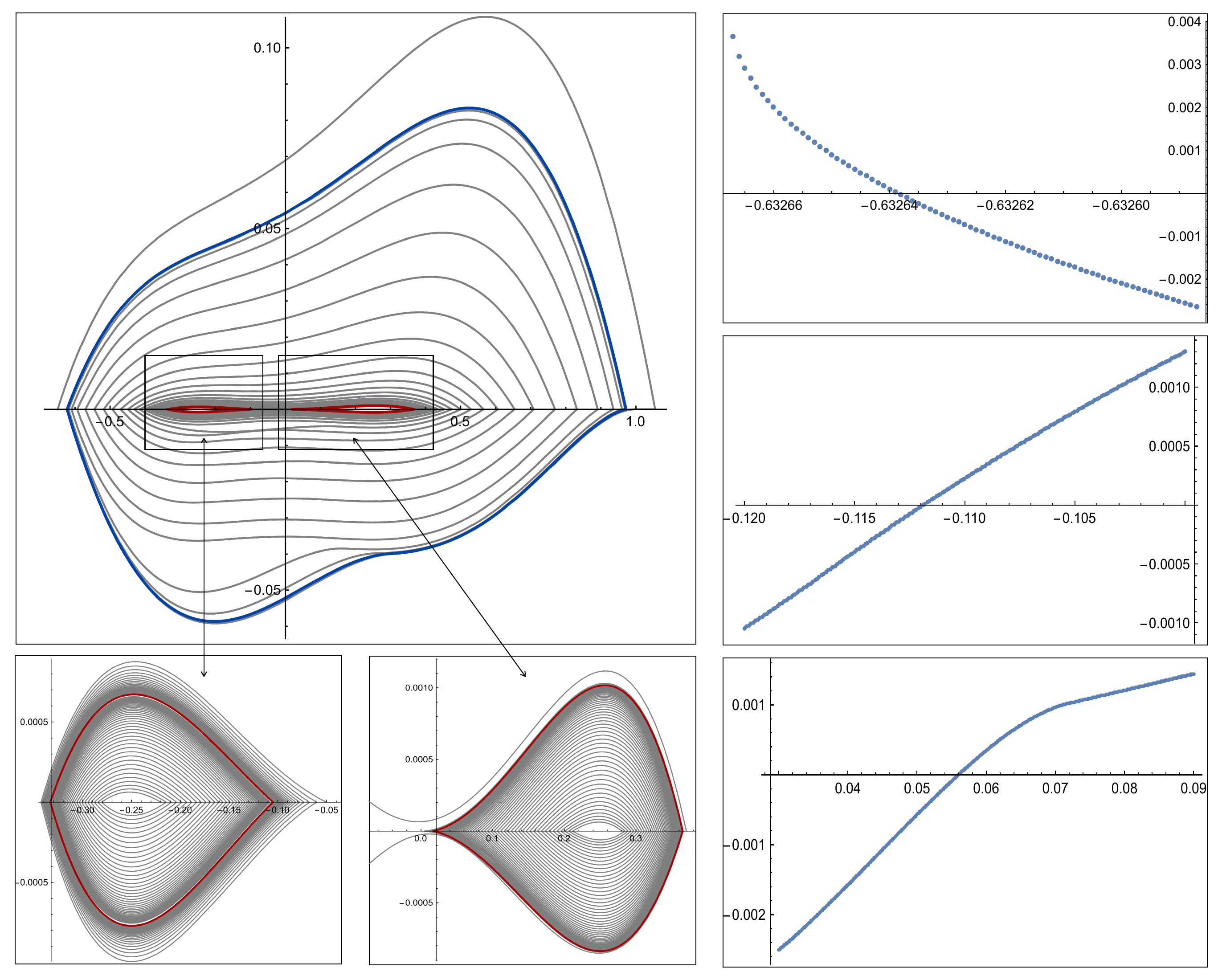}
		\put(3.5,76){\encircle{\emph{A}}}
		\put(23,23){\encircle{\emph{B}}}
		\put(37,23){\encircle{\emph{C}}}
		\put(65,73.5){\encircle{\emph{D}}}
		\put(65,48.4){\encircle{\emph{E}}}
		\put(65,23){\encircle{\emph{F}}}
     	\end{overpic}
     	\vspace{0.5cm}
     	\caption{Numerical simulations were performed in Mathematica to analyze the trajectories of the Filippov system \eqref{eq:example} with parameters $k=2$, $\eta=1/2$, $b=-1/400$, $P^+(x)=-x(x-2)/32$, and $P^-(x)=x(x-1)/16$. Figure (A) illustrates the configuration of three limit cycles where the trajectories converge: the larger limit cycle is asymptotically stable, while the smaller ones are unstable. Figures (B) and (C) provide detailed views of the unstable limit cycles. Figures (D), (E), and (F) display the graph of the displacement function near the asymptotically stable limit cycle, the unstable limit cycle shown in Figure (B), and the unstable limit cycle shown in Figure (C), respectively. Each isolated zero of the displacement function corresponds to a limit cycle of $\widetilde{Z}(x, y)$}
     	\label{fig:simula}
     \end{figure}
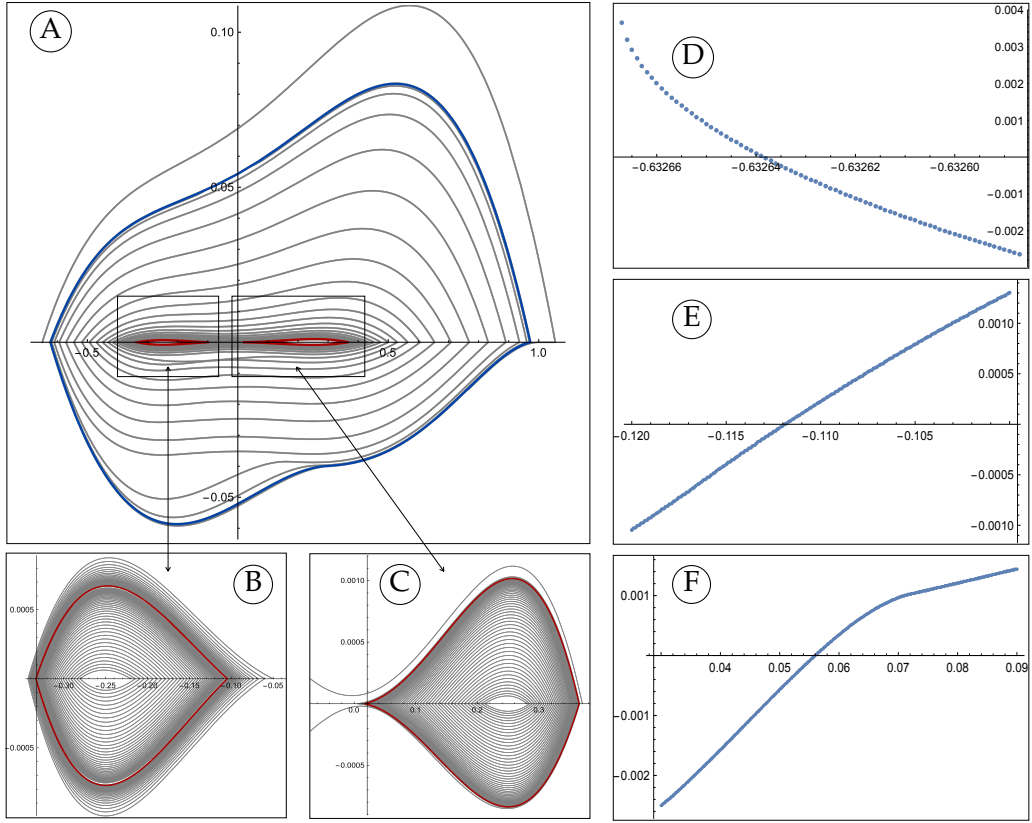

\end{example}

\section{Proof of the main result}
 We start the proof of Theorem \ref{thm:unfold} by considering the following $(2k-2)$-parameter family of perturbations of $Z$:
\begin{equation} \label{sistemainicialperturbado}
	 Z_{\Lambda}(x,y)=  \begin{cases}
		 Z_{\Lambda}^+(x,y)=\left(\begin{array}{c} X^+(x,y)\\ Y_{\Lambda}^+(x,y)\end{array}\right), & y > 0,\vspace{0.3cm} \\
		 Z_{\Lambda}^-(x,y) =\left(\begin{array}{c} X^-(x,y)\\ Y_{\Lambda}^-(x,y) \end{array}\right),& y < 0,
	\end{cases}
\end{equation}
where
\begin{equation}\label{YLambda}
Y_{\Lambda}^{\pm}(x,y)=Y^{\pm}(x,y)+X^{\pm}(x,y)P^{\pm}_{\Lambda}(x),
\end{equation}
 and $P^+_{\Lambda}$, $P^-_{\Lambda}$ are continuous $(2k-2)$-parameter families of polynomials of degree $2k-2$ satisfying $P^+_{O}=P^-_{O}=0$, where $\Lambda=(a_1, \dots, a_{2k-2}) \in \R^{2k-2}$ and $O=(0,\dots,0)\in\R^{2k-2}$. In other words,
 \begin{equation}\label{poly}
 P_{\Lambda}^{\pm}(x)=\sum_{j=1}^{2k-2}c_j^{\pm}(\Lambda) x^j,
 \end{equation}
where $c_j^{\pm}$, $j\in \{1,\ldots,2k-2\}$, are continuous functions vanishing at $\Lambda=O$.

In Subsection \ref{sec:P_Lambda}, under the hypotheses of Theorem \ref{thm:unfold}, we construct the perturbation polynomials $P^+$ and $P^-$ with the aid of the auxiliar polynomials $P^+_{\Lambda}$ and $P^-_{\Lambda}$ as defined in \eqref{poly}. Recall that the polynomials required in Theorem \ref{thm:unfold} must has norm as small as one wants. Thus, to address this norm constraint, we introduce a small parameter $\e > 0$, to be chosen later, and work with the polynomials $P_{\e \Lambda}^{+}(x)$ and $P_{\e \Lambda}^{-}(x)$, whose norms approach zero as $\e \to 0$, provided that the functions $c_j^{\pm}$, for $j \in {1,\ldots,2k-2}$, are continuous and vanish at $\Lambda = 0$ as required. In addition, the parameter functions $c_j^{+}$, $c_j^{-}$, $j\in\{1,\ldots,2k-2\}$, that determine the polynomials $P_{\Lambda}^{+}(x)$ and $P_{\Lambda}^{-}(x)$ in \eqref{poly}, will be obtained in order that the origin and the points $(\e a_i,0)$, $i\in\{1,2,\ldots,2k-2\},$ are tangencies between $\Sigma$ and the vector fields $Z_{\e\Lambda}^+$ and $Z_{\e\Lambda}^-$, for $\e>0$ sufficiently small. In Subsection \ref{sec:mts}, we show that all these tangencies have multiplicity two (see Proposition \ref{prop:multp2}) and that, in addition, $k$ of these tangencies are actually $(2,2)$-monodromic tangential singularities of $Z_{\e\Lambda}$ with non-vanishing second Lyapunov coefficient (see Proposition \ref{prop:22mts}). Finally, in Subsection \ref{sec:proofA}, we conclude that, by considering an additional perturbation, each one of these $(2,2)$-monodromic tangential singularities undergoes a pseudo-Hopf bifurcation, which creates $k$ hyperbolic limit cycle surrounding sliding segments. 
 
\subsection{Construction of $P^+_{\Lambda}$ and $P^-_{\Lambda}$}\label{sec:P_Lambda}
Here, given $\Lambda=(a_1,\ldots,a_{2k-2}) \in \mathcal{L}$, with
\begin{equation}\label{domainL}
\mathcal{L}=\{(a_1, \dots, a_{2k-2}) \in \R^{2k-2}; a_i\neq0\,\,\forall i\,\,\text{and}\,\, a_i \neq a_j\,\forall\,i \neq j\},
\end{equation}
we shall construct polynomials $P^+_{\Lambda}$ and $P^-_{\Lambda}$ as \eqref{poly} satisfying 
\begin{equation} \label{relationshipYi0}
	Y_{\e\Lambda}^+(\e a_i,0)=0\,\, \text{and}\,\, Y_{\e\Lambda}^-(\e a_i,0)=0,\,\, \text{for}\,\, i\in\{1,\ldots,2k-2\}.
\end{equation}
The coefficient $\e$ will be chosen to be sufficiently small later on.

Assuming that the Filippov system \eqref{sistemainicial} has a $(2k, 2k)$-monodromic tangential singularity at the origin (see conditions {\bf C1}, {\bf C2}, and {\bf C3}), we have that $X^{\pm}(0,0)\neq0.$
Therefore, as explained in Subsection \ref{sec:cf}, there exists a small neighborhood $U$ of the origin such that $X^{\pm}(x,y) \neq0$, for all $(x,y) \in U.$  Moreover, based on \eqref{eq:eta} and \eqref{etapm}, one has
\begin{equation}\label{quociente}
\dfrac{Y^{\pm}(x,y)}{X^{\pm}(x,y)}=\pm\delta\,\eta^{\pm}(x,y)= \pm\delta\left(a^{\pm}x^{2k-1} + x^{2k}f^{\pm}(x) + yg^{\pm}(x,y)\right),\,\,\, (x,y)\in U,
\end{equation}
where the values $a^{\pm}$ and the functions $f^{\pm}(x)$ and $g^{\pm}(x,y)$ are given by \eqref{value:a} and \eqref{auxfunc}, respectively. Therefore, by taking into account the expression for $Y^{\pm}_{\Lambda}$ given by \eqref{YLambda}, the relationship \eqref{relationshipYi0} is equivalent to 
\begin{equation}\label{relationship2}
P^{\pm}_{\e\Lambda}(\varepsilon a_i) = \mp\delta\varepsilon^{2k-1}(a^{\pm}  a_i^{2k-1}+\varepsilon a_i^{2k}f^{\pm}(\varepsilon a_i)),\,\,\, i\in\{1,\ldots,2k-2\},
\end{equation}
which is an interpolation problem that can be investigated with the help of a Vandermond matrix. Indeed, by denoting
\begin{equation}\label{eq:Yi}
	\xi_i^{\pm}(\varepsilon) :=   \mp\delta\varepsilon^{2k-1}(a^{\pm}  a_i^{2k-1}+\varepsilon a_i^{2k}f^{\pm}(\varepsilon a_i))
\end{equation}
and taking into account the expression of $P^{\pm}_{\Lambda}$ given by \eqref{poly}, the relationship \eqref{relationship2} becomes
\[
	H(\Lambda, \varepsilon)\left(\begin{array}{ccc}
		c_1^{\pm}(\varepsilon\,\Lambda) & \dots & c_{2k-2}^{\pm}(\varepsilon\,\Lambda)
	\end{array}\right)^T=\left(\begin{array}{ccc}
		\xi_1^{\pm}( \varepsilon)  & \dots & \xi_{2k-2}^{\pm}(\varepsilon)
	\end{array}\right)^T,
\]
where $H(\Lambda, \varepsilon)$ is the following Vandermond Matrix:
\begin{equation}\label{vandermonde}
	H(\Lambda, \varepsilon) = \left(\begin{array}{ccc}
		\varepsilon a_1 &  \dots & \varepsilon^{2k-2}a_1^{2k-2} \\
		\vdots &   & \vdots \\
		\varepsilon a_{2k-2} &  \dots & \varepsilon^{2k-2}a_{2k-2}^{2k-2}, \\
	\end{array}\right),
\end{equation}
which is invertible provided that $a_i \neq a_j$, for $i \neq j$, being ensured by the fact that $\Lambda\in\mathcal{L}$. Then, 
\begin{equation}\label{eq:coefficients}
	\left(\begin{array}{ccc}
		c^{\pm}_1(\e\Lambda) &\dots & c^{\pm}_{2k-2}(\e\Lambda)
	\end{array}\right)^T= [ H(\Lambda, \varepsilon) ]^{-1}\left(\begin{array}{ccc}
		\xi^{\pm}_1( \varepsilon) & \dots & \xi^{\pm}_{2k-2}(\varepsilon)
	\end{array}\right)^T.
\end{equation}
Recall that we require the functions $c_j^{\pm}$, for $j \in {1, \ldots, 2k-2}$, to be continuous and vanish at $\Lambda = O$. To ensure that the relationship \eqref{eq:coefficients} satisfies this requirement, it is crucial to analyze how $\e$ appears in the inverse matrix $[ H(\Lambda, \e) ]^{-1}$, as this may introduce some power of $\e$ in the denominator. For this purpose, we employ the well-known formula for the inverse of the Vandermonde matrix (see, for instance, \cite{knuth97}) to compute the inverse of \eqref{vandermonde} as $[H(\Lambda, \varepsilon)]^{-1} = [b_{ij}]_{2k-2}$, where
\[
		b_{ij} = 
		\Bigg(\dfrac{\sum\limits_{\substack{1 \leq m_1 < \dots < m_{n-i} \leq n \\ m_1, \dots, m_{n-i} \neq j}} (-1)^{j-1} \varepsilon^{2k-2-i}a_{m_1}\dots a_{m_{n-i}}}{\varepsilon a_ j\prod\limits_{\substack{1 \leq m \leq n \\ m \neq j}} (\varepsilon a_m-\varepsilon a_j)} \Bigg) 
		= \varepsilon^{-i} \Bigg(\dfrac{\sum\limits_{\substack{1 \leq m_1 < \dots < m_{n-i} \leq n \\ m_1, \dots, m_{n-i} \neq j}} (-1)^{j-1} a_{m_1}\dots a_{m_{n-i}}}{ a_ j\prod\limits_{\substack{1 \leq m \leq n \\ m \neq j}} ( a_m- a_j)} \Bigg).
\]
This means that the $i$th row of $[H(\Lambda, \varepsilon)]^{-1}$ has order $\varepsilon^{-i}$. Thus, from  \eqref{eq:Yi} and \eqref{eq:coefficients},  $c_i^{\pm}(\e\Lambda)=\e^{2k-1-i}  C_j^{\pm}(\Lambda, \varepsilon)$, where $C_j$ is a smooth function, and so
\begin{equation}\label{polinomioP}
	P^{\pm}_{\e\Lambda}(x) = \sum^{2k-2}_{j=1} \varepsilon^{2k-1-j} C_j^{\pm}(\Lambda, \varepsilon) x^j.
\end{equation}
Notice that the norm of the polynomials $P^{+}_{\e\Lambda}(x)$ and $P^{-}_{\e\Lambda}(x)$  go to zero as $\e\to 0$, as required.

\subsection{Monodromic tangential singularities}\label{sec:mts}
In what follows, we are going to show that, for $\Lambda\in\mathcal{L}$, the Filippov system $Z_{\e\Lambda}$, given by \eqref{sistemainicialperturbado}, has $k$ $(2,2)$-monodromic tangential singularities with non-vanishing second Lyapunov coefficient.

First, in the next result, we will see that, for $\e>0$ sufficiently small, the origin and the points $(0,\e a_i),$   $i\in\{1,\ldots,2k-2\},$ are tangencies of multiplicity $2$ between $\Sigma$ and the vector fields $Z_{\e\Lambda}^+$ and $Z_{\e\Lambda}^-$ of which $k$ of them are invisible. For the sake of simplicity, we will denote $a_0=0$.

\begin{proposition}\label{prop:multp2}
	Let $\e>0$, $a_0=0$, and $\Lambda=(a_1,\ldots,a_{2k-2}) \in \mathcal{L}$, with $\mathcal{L}$ given by \eqref{domainL}. Consider the vector fields $Z_{\e\Lambda}^+$ and $Z_{\e\Lambda}^-$ provided by \eqref{sistemainicialperturbado}. Then, for $\e>0$ sufficiently small, the points $(0,\e a_i),$   $i\in\{0,\ldots,2k-2\},$ are tangencies of multiplicity $2$ between $\Sigma$ and the vector fields $Z_{\e\Lambda}^+$ and $Z_{\e\Lambda}^-$. In addition, if $a_1<0 < a_2  < \dots < a_{2k-2}$,  these tangencies are invisible for $i\in\{1\}\cup\{2,4,\ldots,2k-2\}$ (see Figure \ref{fig1}).
\end{proposition}
\begin{proof}
First of all, notice that the construction of the polynomials $P_{\Lambda}^{\pm}$ implies that the points $(\e a_i,0)$, for $i\in\{0,\ldots,2k-2\}$, are tangencies between $\Sigma$ and the vector fields $Z_{\e\Lambda}^{+}$ and $Z_{\e\Lambda}^{-}$. Indeed, the relationship \eqref{relationshipYi0} implies that $Y_{\e\Lambda}^{\pm}(\e a_i,0)=0$ and, in addition, $X^{\pm}_{\e\Lambda}(\varepsilon a_i, 0) =X^{\pm}(0, 0)+\CO(\e)$ which, by condition {\bf C1}, is non-vanishing for $\e>0$ sufficiently small.

Now, in order to see that these tangencies have multiplicity $2$, consider the function $h^{\pm}(x,\e)=Y^{\pm}_{\e\Lambda} (x,0)$. Notice that,  condition {\bf C1} implies that
\[
h^{\pm}(0,0)=0,\, \dfrac{\partial^i h^{\pm}}{\partial x^i}(0,0)=0\,, \text{for} \, i\in\{1,\ldots,2k-2\},\, \text{and}\, \dfrac{\partial^{2k-1} h^{\pm}}{\partial x^{2k-1}}(0,0)\neq0.
\]
Thus, by Malgrange Preparation Theorem (see, for instance, \cite{malgrange1971}), there exists a neighborhood $V$ of $(x,\e)=(0,0)$ such that $h^{\pm}|_V (x,\e)=A^{\pm}(x,\e)B_{\e}^{\pm}(x)$ where, $A(x,\e)>0$ for $(x,\e)\in V$ and
\[
B_{\e}^{\pm}(x)=\pm\delta a^{\pm} x^{2k-1}+\ell_{2k-2}^{\pm}(\e) x^{2k-2}+\cdots+\ell_1^{\pm}(\e) x+\ell_0^{\pm}(\e).
\]
Notice that $Y_{\e\Lambda}^{\pm}(\e a_i,0)=0$ implies that $B_{\e}^{\pm}(\e a_i)=0$ for all $i\in\{0,\ldots,2k-2\}$. Since $B_{\e}^{\pm}$ is a polynomial in $x$ of degree $2k-1$, we conclude that those roots are simple provided that $a_i\neq a_j$, for all $i\neq j$ in $\{0,\ldots,2k-2\}$, which is ensured by the fact that  $\Lambda=(a_1,\ldots,a_{2k-2})\in\mathcal{L}$ (see \eqref{domainL}). Hence,
\begin{equation}\label{derY}
\dfrac{\partial   Y_{\e\Lambda}^{\pm}}{\partial x}(\varepsilon a_i ,0)=A^{\pm}(\e a_i,\e)(B_{\e}^{\pm})'(\e a_i)\neq0,
\end{equation}
implying that the tangencies $(\e a_i,0)$, for  $i\in\{0,\ldots,2k-2\}$, have multiplicity $2$.

Finally, since $B_{\e}^{\pm}$ has odd degree, its derivative at its smallest root has the same sign as the leading coefficient $\pm\delta a^{\pm}$. Thus, by assuming that $a_1 < 0 < a_2 < a_3 < \dots < a_{2k-2}$ and taking into account that $\sgn(\pm\delta a^{\pm})=\mp 1$ (see \eqref{value:a} and {\bf C2}), it follows that $\sgn((B_{\e}^{\pm})'(\e a_1))=\mp 1$ and, since the derivative at the roots has alternate signs, $\sgn((B_{\e}^{+})'(\e a_{2j}))=\mp1$, for $j=1,\ldots k-1$. Taking \eqref{derY} into account, we conclude that the tangencies $(\e a_i,0)$, for $i\in\{1\}\cup\{2,4,\ldots,2k-1\}$, are invisible  for both vector fields $Z^+_{\e\Lambda}$ and $Z^-_{\e\Lambda}$.The remaining tangencies are visible (see Figure \ref{fig1}).
\end{proof}

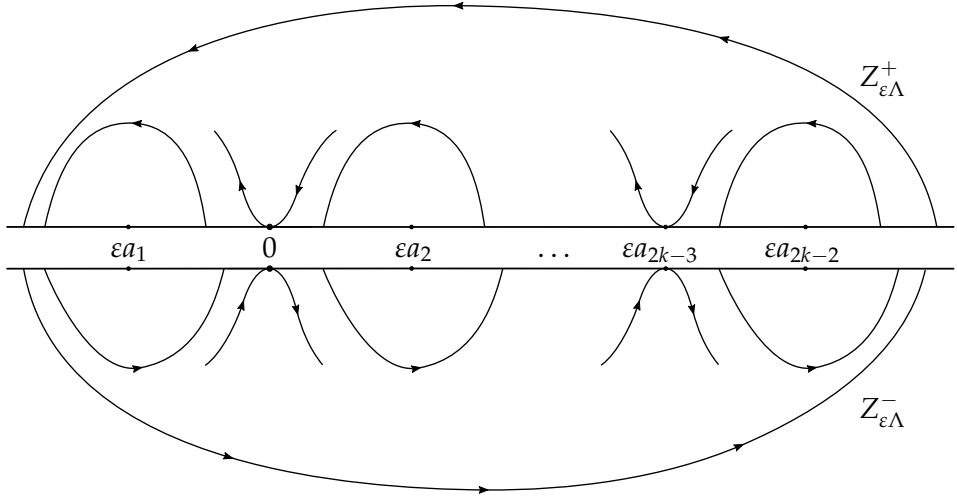
\begin{figure}[H]
	\bigskip
	\centering
	\begin{overpic}[scale=0.6]{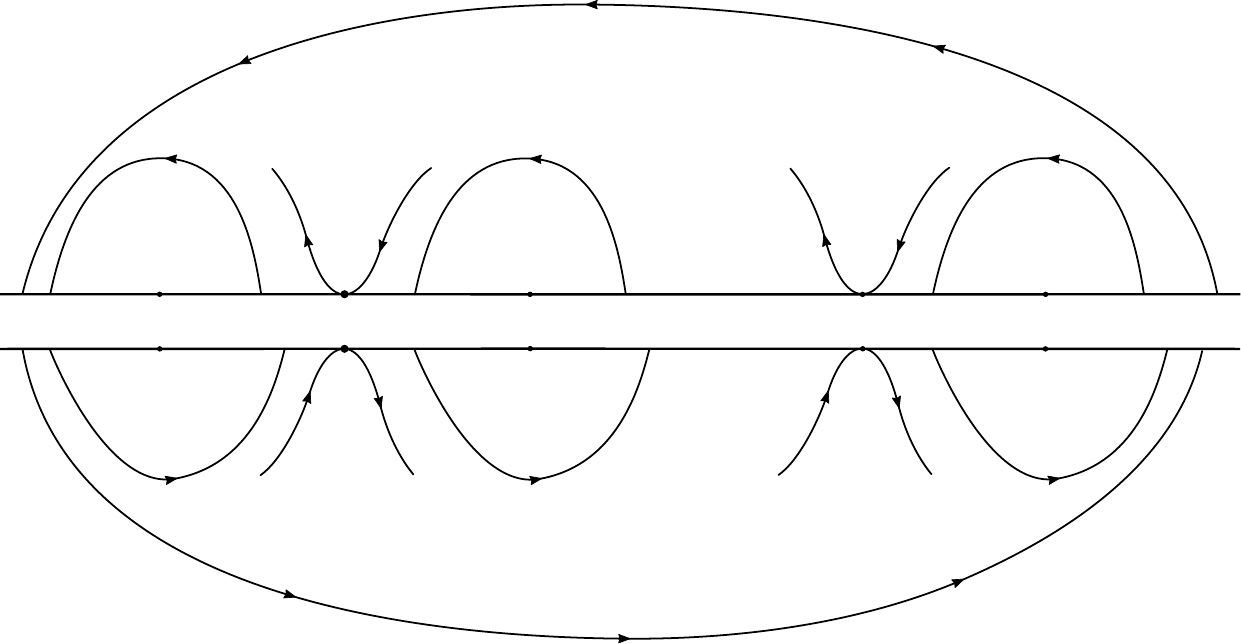}
		\put(90,43){$Z_{\e\Lambda}^+$}
		\put(90,8){$Z_{\e\Lambda}^-$}
		\put(11,25){$\e a_1$}
		\put(27,25){$0$}
		\put(41,25){$\e a_2$}
		\put(56,25){$\dots$}
		\put(65,25){$\e a_{2k-3}$}
		\put(80,25){$\e a_{2k-2}$}
	\end{overpic}
	\vspace{0.8cm}
	\caption{Illustration of the tangencies of multiplicity $2$ between $\Sigma$ and the vector fields $Z_{\e\Lambda}^{+}$ and  $Z_{\e\Lambda}^{-}$ for $\Lambda=(a_1,\ldots,a_{2k-2})\in\mathcal{L}$ with the configuration $a_1 < 0 < a_2 < a_3 < \dots < a_{2k-2}$. They alternate between invisible and visible tangencies.}
	\label{fig1}
\end{figure}

Proposition \ref{prop:multp2} implies that, for $\Lambda\in\mathcal{L}$ and $\e>0$ sufficiently small, the points $(0,\e a_i)$, for  $i\in\{1\}\cup\{2,4,\ldots,2k-1\}$, are $(2,2)$-monodromic tangential singularities of the Filippov system  $Z_{\e\Lambda}$ given by \eqref{sistemainicialperturbado}. In what follows, we are going to compute the second Lyapunov coefficient for each one of these singularities.

\begin{proposition}\label{prop:22mts}
Let $\e>0$ and $\Lambda=(a_1,\ldots,a_{2k-2})\in\mathcal{L},$ with $\mathcal{L}$ given by \eqref{domainL}, satisfying $a_1 < 0 < a_2 < a_3 < \dots < a_{2k-2}$. Consider the Filippov system $Z_{\e\Lambda}$ provided in \eqref{sistemainicialperturbado}. Then, for each $i\in\{1\}\cup\{2,4,\ldots,2k-2\}$ and $\e>0$ sufficiently small, the second Lyapunov coefficient $V_{2,i}(\e)$ associated with the $(2,2)$-monodromic tangential singularity $(\e a_i,0)$ satisfies
\begin{equation}\label{V2i}
V_{2,i}(\e)=\dfrac{2k+1}{3}V_2+\CO(\e).
\end{equation}
\end{proposition}
\begin{proof}
We start by fixing an index $i\in\{1\}\cup\{2,4,\ldots,2k-2\}$ and translating the point $(\e a_i,0)$ to the origin by means of the change of coordinates $u=x-\varepsilon a_i$.  Thus, the vector field $Z_{\e\Lambda}$, given by \eqref{sistemainicialperturbado}, becomes 
\begin{equation}\label{Ze-proof}
	\widetilde Z_{\e}(u,y)= \begin{cases}
		 \left(\begin{array}{c} \widetilde{X}_{\e}^+(u,y),\\ \widetilde{Y}_{\e}^+(u,y)\end{array}\right), & y > 0,\vspace{0.3cm} \\
		 \left(\begin{array}{c} \widetilde{X}_{\e}^-(u,y),\\ \widetilde{Y}_{\e}^-(u,y)\end{array}\right),& y < 0,
	\end{cases}
\end{equation}
where
\[
\widetilde{X}_{\e}^{\pm}(u,y)=X^{\pm}(u+\varepsilon a_i,y)\,\,\, \text{and} \,\,\,
\widetilde{Y}_{\e}^{\pm}(u,y)=Y^{\pm}(u+\varepsilon a_i,y)+X^{\pm}(u+\varepsilon a_i,y)P^{\pm}_{\e\Lambda}(u+\varepsilon a_i).
\]
Hence, the second Lyapunov coefficient of the origin of \eqref{Ze-proof} is given by  \eqref{eq:coefV2NP} as
\begin{equation}\label{eq:coefV2P}
V_{2,i}(\e)=\delta(\alpha_{2,i}^+(\e)-\alpha_{2,i}^-(\e)), \,\,\, \alpha_{2,i}^{\pm}(\e)=\dfrac{-2\tilde f^{\pm}_{0,\e}\pm 2 \delta \tilde a^{\pm}_{\e} \tilde g^{\pm}_{0,0,\e}}{3 \tilde a^{\pm}_{\e}},
\end{equation}
where $\delta=\sgn(X^+_{\e}(0,0))=\sgn(X^+(0,0)),$ for $\e>0$ sufficiently small,
\[
\tilde a^{\pm}_{\e}=\dfrac{1}{|\widetilde{X}^{\pm}_{\e}(0,0)|}\dfrac{\partial \widetilde{Y}^{\pm}_{\e}}{\partial u}(0,0),\,\,\, \tilde f^{\pm}_{0,\e}=\tilde f^{\pm}_{\e}(0),\,\,\,\text{and}\,\,\, \tilde g^{\pm}_{0,0,\e}=\tilde g^{\pm}_{\e} (0,0),
\]
with
\[
\begin{aligned}
		\tilde f^{\pm}_{\e}(u)=&\dfrac{\pm \delta \widetilde{Y}^{\pm}_{\e}(u,0)-\tilde a^{\pm}_{\e}u \widetilde{X}^{\pm}_{\e}(u,0)}{u^{2} \widetilde{X}^{\pm}_{\e}(u,0)}\,\,\, \text{and}\vspace{0.2cm}\\
		\tilde{g}^{\pm}_{\e}(u,y)=&\dfrac{\pm \widetilde{X}^{\pm}_{\e}(u,0)\widetilde{Y}^{\pm}_{\e}(u,y) \mp \widetilde{X}^{\pm}_{\e}(u,y)\widetilde{Y}^{\pm}_{\e}(u,0)}{y \delta \widetilde{X}^{\pm}_{\e}(u,y)\widetilde X^{\pm}_{\e}(u,0)}.
	\end{aligned}
\]
Using that (see \eqref{quociente})
\[
Y^{\pm}(x,y)= \pm\delta X^{\pm}(x,y)\left(a^{\pm}x^{2k-1} + x^{2k}f^{\pm}(x) + yg^{\pm}(x,y)\right),\,\,\, (x,y)\in U,
\]
 we compute
\begin{equation}\label{eq:tildefg}
	\begin{array}{rl}
		\tilde a^{\pm}_{\e} = & \pm \delta (P^{\pm}_{\e\Lambda})'(\varepsilon a_i)+\varepsilon^{2 k-2}(2 k  -1)a^{\pm}a_i^{2 k-2}+2 \varepsilon^{2 k-1} k a_i^{2 k-1}  f^{\pm}(\varepsilon a_i  )\vspace{0.2cm}\\
		&+ \varepsilon^{2 k} a_i^{2 k}(f^{\pm})'(a_i \varepsilon ),\vspace{0.2cm}\\
		\tilde f^{\pm}_{0,\e}= & \dfrac{\pm\delta (P^{\pm}_{\e\Lambda})''(\varepsilon a_i)}{2}  +\e^{2 k-3} a^{\pm} (k-1) (2 k-1) a_i^{2 k-3}+\e^{2 k-2} k (2 k-1) a_i^{2 k-2} f^{\pm}(a_i \e  ) \vspace{0.2cm}\\
		&+2\e^{2 k-1}  k a_i^{2 k-1} (f^{\pm})'(a_i \e  ) +\e^{2 k} \dfrac{a_i^{2 k} (f^{\pm})''(a_i \e  )}{2},\vspace{0.3cm}\\
		\tilde{g}^{\pm}_{0,0,\e} = &g^{\pm}(\varepsilon a_i,0).
	\end{array}
\end{equation}

From \eqref{polinomioP}, we have that
\[
(P_{\e\Lambda}^{\pm})'(\varepsilon a_i)  = \sum_{j=1}^{2k-2}j a_i^{j-1} \varepsilon^{2k-2}C^{\pm}_j(\Lambda, \varepsilon)\,\,\,\text{and}\,\,\,  	(P_{\e\Lambda}^{\pm})''(\varepsilon a_i)  = \sum_{j=2}^{2k-2}j(j-1)a_i^{j-2} \varepsilon^{2k-3}C^{\pm}_j(\Lambda, \varepsilon).
\]
Thus, by denoting
\begin{equation}\label{defsums}
	\begin{aligned}
		s_1^{\pm} & = \sum_{j=1}^{2k-2}ja_i^{j-1}  C_j^{\pm}(\Lambda,0), \\ 
		s_2^{\pm} & = \sum_{j=1}^{2k-2}ja_i^{j-1}  \dfrac{\partial C_j^{\pm}}{\partial \e}(\Lambda, 0), \\ 
		s_3^{\pm} & = \sum_{j=2}^{2k-2}j(j-1)a_i^{j-2}  C_j^{\pm}( \Lambda,0), \\ 
		s_4^{\pm} & = \sum_{j=2}^{2k-2} j(j-1)a_i^{j-2}  \dfrac{\partial C_j^{\pm}}{\partial \e}( \Lambda,0),
	\end{aligned}
\end{equation}
we get that
\begin{equation}\label{eq:P10epsilonaux}
	P_{\Lambda}'(\varepsilon a_i)  = \varepsilon^{2k-2} s^{\pm}_1+ \varepsilon^{2k-1} s^{\pm}_2 + O(\varepsilon^{2k})  \,\,\,\text{and}\,\,\, 	P_{\Lambda}^{''}(\varepsilon a_i)  = \varepsilon^{2k-3} s^{\pm}_3+  \varepsilon^{2k-2} s^{\pm}_4 + O(\varepsilon^{2k-1}).
\end{equation}
 
Substituting \eqref{eq:P10epsilonaux} into \eqref{eq:tildefg} we obtain
 \begin{equation}\label{eqaux:tildefg}
	\begin{aligned}
		\tilde a^{\pm}_{\e} = &  \varepsilon^{2 k-2}\Big((2 k  -1)a^{\pm}a_i^{2 k-2} \pm \delta s^{\pm}_1  \Big)\\&+ \varepsilon^{2 k-1}\Big(2 k a_i^{2 k-1}  f^{\pm}(\varepsilon a_i  ) \pm \delta s^{\pm}_2\Big)+\CO(\e^{2k}),\vspace{0.3cm}\\
		\tilde f^{\pm}_{0,\e}= &  \e^{2 k-3} \left(a^{\pm} (k-1) (2 k-1) a_i^{2 k-3}\pm\dfrac{  \delta s^{\pm}_3}{2}\right)\\&+\e^{2 k-2}\left( k (2 k-1) a_i^{2 k-2} f^{\pm}(a_i \e  ) \pm\dfrac{\delta s^{\pm}_4}{2} \right)+\CO(\e^{2k-1}),\vspace{0.3cm}\\
		\tilde{g}^{\pm}_{0,0,\e} = &g^{\pm}(\varepsilon a_i,0).
	\end{aligned}
 \end{equation}
 
From here, in order to compute $V_{2,i}(\e)$, we proceed with the following steps. Firstly, we substitute the expressions in \eqref{eqaux:tildefg} into the formula \eqref{eq:coefV2P} to get
\[
\begin{aligned}
\alpha^{\pm}_{2,i}(\e)=&\dfrac{\mp\delta a_i^3s_3^{\pm}-(2k-2)(2k-1)a^{\pm}a_i^{2k}+\e \Big((2 g^{\pm}_{0,0}s^{\pm}_1 \mp\delta s^{\pm}_4)a_i^3 - (4 k-2) (f^{\pm}_{0} k\mp\delta a^{\pm}  g^{\pm}_{0,0})a_i^{2 k+1} \Big)+\CO(\e^2)}
{\e  \left(3(2 k- 1)a^{\pm}a_i^{2 k+1}\pm3 \delta a_i^3  s^{\pm}_1\right)+\e ^2 \left(6 k f^{\pm}_{0} a_i^{2 k+2}\pm 3 \delta a_i^3  s^{\pm}_2\right)+\CO(\e^3)}\\
=&\e^{-1}\frac{\mp \delta a_i^3  s^{\pm}_3-(2k-2) (2 k-1) a^{\pm} a_i^{2 k}}{3 a^{\pm} (2 k-1) a_i^{2 k+1}\pm 3 \delta a_i^3  s^{\pm}_1}+A^{\pm}_i+\CO(\e).
\end{aligned}
\]
The expression of $A^{\pm}_i$ is a little cumbersome, so we shall omit it here. Secondly, taking into account the  definition of the coefficients of the polynomials $P^{\pm}_{\Lambda}$ in \eqref{eq:coefficients} and expression \eqref{eq:Yi}, we get the following relationships  
\[
			s_1^{-}  = - \dfrac{a^{-}}{a^{+}} s_1^+,  \,\, s_2^{-}  = - \dfrac{f^{-}_0}{f^{+}_0} s_2^+,  \,\,s_3^{-}  = - \dfrac{a^{-}}{a^{+}} s_3^+, \text{and}  \,\,s_4^{-}  =  -\dfrac{f^{-}_0}{f^{+}_0} s_4^+,
\]
which implies that the coefficients of $\e^{-1}$ in the expansions above for $\alpha^{+}_{2,i}(\e)$ and $\alpha^{-}_{2,i}(\e)$  coincide. Thus, from \eqref{eq:coefV2P}, we have that 
\[
V_{2,i}(\e)=\delta(A_i^+-A_i^-)+\CO(\e),
\]
implying that $V_{2,i}(\e)$ is continuous at $\e=0$. Moreover, the relationships above allow to get rid of the terms $s_1^-$, $s_2^-$, $s_3^-$, and $s_4^-$ in $A_i^-$.
Finally, by using the relationships 
\begin{equation}\label{proofrel}
		\begin{aligned}
			s_2^{+} & =  \dfrac{f^{+}_0}{a^{+}} \left[ (a_i-\alpha)s_1^{+} - \delta a^{+} a_i^{2k-1} -\delta (2k-1) a^{+} \alpha a_i^{2k-2} \right] \,\,\, \text{and}\\ 
			s_4^{+} & =  \dfrac{f^{+}_0}{a^{+}} \left[ (a_i-\alpha)s_3^{+}+ 2 s_1^{+} -\delta (2k-2)(2k-1) a^{+} \alpha a_i^{2k-3} \right],
			\end{aligned}
		\end{equation}
provided by Lemma \ref{lemma1} of Appendix B, and taking into account that $\delta^{2n}=1$ and $\delta^{2n+1}=\delta$ for any integer $n$, we obtain
\[
V_{2,i}(\e)=\dfrac{2}{3}\left( \dfrac{a^+ g^+_{0,0}-\delta f^+_0}{a^+} +\dfrac{a^- g^-_{0,0}+\delta f^-_0}{a^-}\right)+\CO(\e)=\dfrac{2k+1}{3}V_2+\CO(\e).
\]
It concludes this proof.
\end{proof}

\subsection{Proof of Theorem \ref{thm:unfold}}\label{sec:proofA}

Consider the Filippov system  $Z_{\e\Lambda}$, given by \eqref{sistemainicialperturbado}, and let $U\subset\R^2$ be a neighborhood of  $(0,0)$ and $\lambda>0$.

From propositions \ref{prop:multp2} and \ref{prop:22mts}, we can fix $\Lambda=(a_1,\ldots,a_{2k-2})\in\mathcal{L}$, with $a_1<0<a_2<\ldots a_{2k-2}$, such that, for $\e>0$ sufficiently small, the Filippov system $Z_{\e\Lambda}$ has $k$ $(2,2)$-monodromic tangential singularities, namely, the points $(\e a_i,0)$ for $i\in\{1\}\cup\{2,4,\ldots,2k-2\}$. Moreover, the second Lyapunov coefficients of each one of these singularities writes like \eqref{V2i} and, therefore, for $\e>0$ sufficiently small, they are all non-vanishing and have the same sign as $V_2$. Thus, we fix $\e^*>0$ such that the norms of the polynomials $P^+_{\e^* \Lambda}$ and $P^-_{\e^* \Lambda}$ are smaller than $\lambda$, $(\e^* a_i,0)\in U$ and $\sgn(V_{2,i}(\e^*))=\sgn(V_2)$ for $i\in\{1\}\cup\{2,4,\ldots,2k-2\}$.

Now, denote $P^{+}(x)=P^+_{\e^* \Lambda}$ and $P^{-}(x)=P^-_{\e^* \Lambda}$ and consider the one-parameter family of Filippov systems $\widetilde Z_b(x,y)$, given by \eqref{teo:sistemainicialperturbado}.
For each $i\in\{1\}\cup\{2,4,\ldots,2k-2\}$, consider the translated Filippov system $Z_b^i(u,x)=\widetilde Z_b(u+\e^* a_i,y)$. We notice that $Z_b^i$ writes like \eqref{Zb} and satisfies all the hypotheses of Proposition \ref{prop:phd1}. Thus, given a neighborhood $U_i\subset\R^2$  of the origin  satisfying $U_i+(\e^*a_i,0)\subset U$, there exists an interval $I_i\subset \R$ containing $0$ such that, for all $b\in I_i$ satisfying $\sgn(b)=-\sgn(\delta V_2)$, the Filippov system $Z_b^i$ has a hyperbolic limit cycle inside $U$ surrounding a sliding segment. In addition, the hyperbolic limit cycle is stable (resp. unstable) provided that $V_{2}<0$ (resp. $V_{2}>0$). Since the neighborhoods $U_i$, $i\in\{1\}\cup\{2,4,\ldots,2k-2\}$, can be arbitrarily chosen, we can impose that $(U_i+\e^* a_i)\cap (U_j+\e^* a_j)=\emptyset$ for $i\neq j$ in $\{1\}\cup\{2,4,\ldots,2k-2\}$.

Hence, by taking  $I=I_1\cap (I_2\cap I_4\cap\cdots \cap I_{2k-2})$, which is an interval containing the origin, we conclude that for each $i\in\{1\}\cup\{2,4,\ldots,2k-2\}$ and for every $b\in I$ satisfying $\sgn(b)=-\sgn(\delta V_2)$, the Filippov vector $\widetilde Z_b(x,y)$ has a limit cycle contained in $U_i+(\e^*a_i,0)\subset U$ surrounding an sliding segment.

\section*{Appendix A: Pseudo-Hopf bifurcation for monodromic tangential singularities}

In this appendix, we prove a degenerate version of the Pseudo-Hopf bifurcation for $(2k^+,2k^-)$-monodromic tangential singularities, which is characterized by the birth of a hyperbolic limit cycle when a sliding segment is created provided that there exists a first non-vanishing Lyapunov coefficient $V_{2\ell}$ (for a recursive formula for all the Lyapunov coefficients, see \cite[Theorem C]{novaessilva2020}).  Such a bifurcation phenomenon is briefly commented in \cite[Remark 1]{novaessilva2020}.  However, we have noticed that in the literature a proof for this result is only given for $k^+=k^-=\ell=1$, which is precisely the statement of Proposition \ref{prop:phd1}.

\begin{proposition} \label{prop:pseudohopfgeral}
Assume that the Filippov system $Z$, given by  \eqref{sistemainicial}, has a $(2k^+,2k^-)$-monodromic tangential singularity at the origin. Let $V_{2 \ell}$ be its first non-vanishing Lyapunov coefficient. 
Consider the 1-parameter family of Filippov systems
\[
Z_{ b}(x,y)=\begin{cases}
Z^+(x- b,y),&y>0,\\
Z^-
(x,y),&y<0.
 \end{cases}
\]
Then, given a neighborhood $U\subset \R^2$ of $(0,0)$, there exists a neighborhood $I\subset\R$ of $0$ such that the following statements hold for every $ b\in I$:
\begin{itemize}
\item If $\sgn( b)=-\sgn(\delta V_{2 \ell})$, then $Z_{ b}$ has a hyperbolic limit cycle in $U$ surrounding a sliding segment. In addition, the hyperbolic limit cycle is stable (resp. unstable) provided that $V_{2 \ell}<0$ (resp. $V_{2 \ell}>0$).
\item If $\sgn( b)=\sgn(\delta V_{2 \ell})$, then $Z_{ b}$ does not have limit cycles in $U$.
 \end{itemize}
 \end{proposition}
\begin{proof}
First, let $\Delta_0(x)$ be the displacement function of $Z$ defined in a neighborhood of $x=0$. From \eqref{expansion}, we known that
\[
\Delta_0(x)=V_{2\ell}x^{2\ell}+\CO(x^{2\ell+1}).
\]

Now, recall that $\Delta(x;b):=\delta(\varphi_b^+(x)-\varphi^-(x))$ (see \eqref{def:disp}), where $\varphi_b^+$ and $\varphi^-$ are the half-return maps of $Z^+(x-b,y)$ and $Z^-(x,y)$ associated with $\Sigma$. It is easy to see that if $\varphi^+$ is the half-return map of $Z^+(x,y)$, then $\varphi_b^+(x)=\varphi^+(x-b)+b.$ Therefore, taking into account that $(\f^+)'(0)=-1$, we get
\[
 	\begin{aligned}
 	\Delta(x;b)&=\delta(\varphi_b^+(x)-\varphi^-(x))\\
		&=\Delta_0(x) + 2\delta b + b O(x) + O(b^2) \\
 	& = V_{2\ell}x^{\ell} + 2 \delta b + O(x^{2 \ell +1}) +b O(x)+O(b^2).
 	\end{aligned}
\]

Notice that solutions of 
\begin{equation}\label{Deltab0}
\Delta_b(x)=0,\,\,\, x>|b|,
\end{equation}
correspond to crossing periodic solutions of $Z_b$. In addition, simple solutions of \eqref{Deltab0} correspond to hyperbolic limit cycles of $Z_b$.

Denote
\begin{equation}\label{Deltatilde}
	\begin{aligned}
		\widetilde \Delta(y;\beta  ):= \dfrac{\Delta(\beta   y; \mu \beta  ^{2\ell}  )}{\beta  ^{2 \ell}},
	\end{aligned}
\end{equation}
 where $\mu=-\sgn(\delta V_{2 \ell})$. Notice that
\begin{equation}\label{expansionDeltatilde}
	\begin{aligned}
		\widetilde \Delta(y;\beta  ) = & \dfrac{V_{2 \ell} \beta  ^{2 \ell} y^{2\ell}+2\delta \mu \beta  ^{2\ell}+ O(\beta ^{2\ell+1}y^{2\ell+1}) + \mu \beta ^{2\ell}O(\beta  y)}{ \beta ^{2\ell}}\\
		 = & V_{2 \ell}  y^{2\ell}+2\delta \mu + O(\beta  y^{2\ell+1}) + O(\beta  y) \\
		  = & V_{2 \ell}  y^{2\ell}+2\delta \mu + O(\beta ).
	\end{aligned}
\end{equation}
Then, for 
\begin{equation}\label{y0}
y_0 =  \sqrt[2\ell]{\left|\dfrac{2\delta}{V_{2\ell}}\right|},
\end{equation} we have
\[
\widetilde \Delta(y_0,0)=0\,\,\,\text{and}\,\,\,	\dfrac{\partial \widetilde \Delta}{\partial y}(y_0,0)= \pm 2\ell V_{2 \ell}  \left( \sqrt[2\ell]{\left|\dfrac{2\delta}{V_{2\ell}}\right|}\right)^{2\ell-1} \neq 0.
\]
Then, from the Implicit Function Theorem, there exists $y(\beta )$ such as $y(0)=y_0$ and $\widetilde \Delta(y(\beta );\beta )= 0$ for every $\beta $ in a neighborhood of $0$.

Therefore, from  \eqref{Deltatilde},
$	\Delta(\beta  y(\beta ); \mu \beta ^{2\ell}  )=\beta ^{2 \ell}\widetilde \Delta(y(\beta );\beta )= 0.
$
Then, by taking $\beta=\sqrt[2\ell]{\mu b}$, we have that $x(b )= \sqrt[2\ell]{\mu b}\,y(\sqrt[2 \ell]{\mu b}) $ is a solution of  $\Delta(x;b) = 0$ for $\mu b>0$, that is $\sgn( b)=\sgn(\mu)=-\sgn(\delta V_{2 \ell})$. In this case, 
$x(b)=\sqrt[2\ell]{\mu b} \,y_0+\CO(b)$ with $y_0\neq0$ give  in \eqref{y0} and, therefore, $x(b)>|b|$ for $|b|\neq0$ sufficiently small. Hence $x(b)$ corresponds to a crossing periodic solution of $Z_b$.

In addition, from \eqref{expansionDeltatilde},
\begin{equation}\label{derivativeDeltatilde}
	\begin{aligned}
		\dfrac{1}{b^{\frac{2\ell - 1}{2\ell}}}\dfrac{\partial \Delta}{\partial x}(x(b);b)= & \dfrac{2\ell V_{2 \ell}  \left(\sqrt[2\ell]{\mu b}y(\sqrt[2 \ell]{\mu b})\right)^{2\ell-1} + O(x^{2\ell}) + b O(1) + O(b ^2)}{b^{\frac{2\ell - 1}{2\ell}}} \\
		=& 2\ell V_{2\ell}O(b^{\frac{2\ell - 1}{2\ell}}) + O(b ) \neq 0
	\end{aligned}
\end{equation}
Therefore, the periodic solution associated to $x(b)$ is actually a limit cycle which is contained in $U$ for $|b|$ small enough.

Finally, from \eqref{derivativeDeltatilde},
		$\text{sign}\left(\dfrac{\partial\Delta}{\partial x}(x(b);b)\right) = \text{sign}(V_{2\ell}),$
	then the hyperbolic limit cycle is stable (resp. unstable) provided that $V_{2\ell}<0$ (resp. $V_{2\ell}>0$), which finishes the proof.
\end{proof}

\section*{Appendix B: Useful relationships}

In this appendix, we prove the relationships \eqref{proofrel} used in the proof of Theorem \ref{thm:unfold}.

\begin{lemma}\label{lemma1} Consider the values defined in \eqref{defsums}. Then, the following relationships holds:
	\begin{equation}\label{rel-lemma2}
		\begin{aligned}
			s_2^{\pm} & =  \dfrac{f^{\pm}_0}{a^{\pm}} \left[ (a_i-\alpha)s_1^{\pm} \mp \delta a^{\pm} a_i^{2k-1} \mp\delta (2k-1) a^{\pm} \alpha a_i^{2k-2} \right] ,\\ 
			s_4^{\pm} & =  \dfrac{f^{\pm}_0}{a^{\pm}} \left[ (a_i-\alpha)s_3^{\pm}+ 2 s_1^{\pm} \mp\delta (2k-2)(2k-1) a^{\pm} \alpha a_i^{2k-3} \right],
		\end{aligned}
	\end{equation}
	where 
	\begin{equation}\label{eqf:alpha}
	\alpha=- \sum_{j=1}^{2k-2} a_j.
\end{equation}
\end{lemma}
\begin{proof}

First, from \eqref{polinomioP}, we have
	\[\label{eqaux:polyP2}
\dfrac{{P}_{\Lambda}^{\pm}(\varepsilon a_i)}{\varepsilon^{2k-1}} = \sum_{j=1}^{2k-2}C_j^{\pm}(\Lambda,\varepsilon)a_i^j.
	\]
	Thus, from  \eqref{relationship2}, the following relationship hold for every $\e>0$ sufficiently small:
	\begin{equation}\label{eqaux:polyP3}
		\begin{aligned}
			\sum_{j=1}^{2k-2}C_j^{\pm}(\Lambda,\varepsilon)a_i^j=\mp\delta a_i^{2k-1}(a^{\pm} + a_i\varepsilon f^{\pm}(\varepsilon a_i) ).
		\end{aligned}
	\end{equation}
	
	Now, define the polynomial
	\begin{equation}\label{eq:T}
		T^{\pm}_{\Lambda}(x) = \sum_{j=1}^{2k-2}C^{\pm}_j(\Lambda, 0)x^j \pm\delta a^{\pm} x^{2k-1}.
	\end{equation}
Notice that $\text{deg}(T^{\pm}_{\Lambda})=2k-1$ and, by taking $\e=0$ in \eqref{eqaux:polyP3}, we get
	\[T^{\pm}_{\Lambda}(0)=  0\,\,\,\text{and}\,\,\, T^{\pm}_{\Lambda}(a_i) = 0,\,\,\, \text{for}\,\,\,  i\in\{1, \dots,2k-2\}.\]
	Thus, the polynomial $T^{\pm}_{\Lambda}$ can be factorized as follows:
	\begin{equation}\label{eqaux:T}
		T^{\pm}_{\Lambda}(x)= \pm\delta a^{\pm}x(x-a_1)(x-a_2)\dots (x-a_{2k-2}).
	\end{equation}
	
Also, define the polynomial
\begin{equation}\label{eq:U}
	U^{\pm}_{\Lambda}(x)=\sum_{j=1}^{2k-2}\dfrac{\partial C^{\pm}_j}{\partial \e}(\Lambda, 0)x^j\pm\delta f^{\pm}_0x^{2k}.
\end{equation}
Notice that $\text{deg}(U^{\pm}_{\Lambda}) =2k$ and,  by taking the derivative of \eqref{eqaux:polyP3} at $\e=0$, we get
\[U^{\pm}_{\Lambda}(0)=0,  U^{\pm}_{\Lambda}(a_i)=0, \text{for}  i\in\{1, 2, \dots, 2k-2\}.\]
Thus, the polynomial $U^{\pm}_{\Lambda}$ can be factorized as follows:
\begin{equation}\label{eqaux:U}
	U^{\pm}_{\Lambda}(x)= \pm\delta f^{\pm}_0x(x-\alpha^{\pm})(x-a_1)(x-a_2)\dots (x-a_{2k-2}),
\end{equation}
where $\alpha^{\pm}$ is, a priori, an unknown root of $U^{\pm}_{\Lambda}$. Nevertheless, since the coefficient of the monomial $x^{2k-1}$ of $U^{\pm}_{\Lambda}$ is zero, we know that the sum of its roots must vanish, which provides that $\alpha^+=\alpha^-=\alpha$, where $\alpha$ is given in \eqref{eqf:alpha}.

From the equations \eqref{eqaux:T} and \eqref{eqaux:U}, we have 
\[
U^{\pm}_{\Lambda}(x)=\dfrac{f^{\pm}_0}{a^{\pm}}(x-\alpha)T^{\pm}_{\Lambda}(x)
\]
and, therefore, from \eqref{eq:T} and \eqref{eq:U}, we get the following recursive relationships:
\begin{equation}\label{eq:dercoefci}
	\begin{aligned}
	\dfrac{\partial C^{\pm}_{1}}{\partial \e}(\Lambda,0)  & = -\dfrac{f^{\pm}_0}{a^{\pm}} \alpha C^{\pm}_1(\Lambda,0), \\
	\dfrac{\partial C^{\pm}_{j}}{\partial \e}(\Lambda,0) & = \dfrac{f^{\pm}_0}{a^{\pm}} ( C_{j-1}(\Lambda,0)  - \alpha C_j (\Lambda,0) ), \,\,\, 2\leq j\leq 2k-2,\\
	C_{2k-2}^{\pm}(\Lambda,0) &=\pm\delta a^{\pm}\alpha .
	\end{aligned}
\end{equation}

For the sake  of simplicity, in what follows, we denote $C^{\pm}_{i}=C^{\pm}_{i}(\Lambda, 0).$ By using \eqref{defsums}, \eqref{eq:dercoefci}, and \eqref{eqaux:polyP3}, we obtain $s_2$ as follows:
\[
	\begin{aligned}
		s_2^{\pm} &  =  \dfrac{\partial C^{\pm}_{1}}{\partial \e}(\Lambda,0)+ \sum_{j=2}^{2k-2} a_i^{j-1} j \dfrac{\partial C^{\pm}_{j}}{\partial \e}(\Lambda,0) \\
		& = \dfrac{f^{\pm}_0}{a^{\pm}}\left[  -\alpha C^{\pm}_1 - \alpha \sum_{j=2}^{2k-2} a_i^{j-1} j C^{\pm}_j +\sum_{j=2}^{2k-2} a_i^{j-1} j C^{\pm}_{j-1} \right]\\
		& =\dfrac{f^{\pm}_0}{a^{\pm}}\left[  -\alpha s^{\pm}_1 + \sum_{j=2}^{2k-2} a_i^{j-1} j C^{\pm}_{j-1} \right]\\
		& = \dfrac{f^{\pm}_0}{a^{\pm}}\left[ -\alpha  s^{\pm}_1 +  \sum_{j=1}^{2k-3} a_i^{j} (j+1) C^{\pm}_{j} \right]\\
		& =\dfrac{f^{\pm}_0}{a^{\pm}}\left[  -\alpha  s^{\pm}_1 +  \sum_{j=1}^{2k-2} a_i^{j} (j+1) C^{\pm}_{j} -a_i^{2k-2}(2k-1)C^{\pm}_{2k-2}\right]\\
		& =\dfrac{f^{\pm}_0}{a^{\pm}}\left[  -\alpha  s^{\pm}_1 +  \sum_{j=1}^{2k-2} a_i^{j} j C^{\pm}_{j}+\sum_{j=1}^{2k-2} a_i^{j} C^{\pm}_{j} -a_i^{2k-2}(2k-1)C^{\pm}_{2k-2} \right]\\
		& =\dfrac{f^{\pm}_0}{a^{\pm}}\left[  -\alpha  s^{\pm}_1 +  a_i s^{\pm}_1\mp \delta a^{\pm} a_i^{2k-1} \mp\delta (2k-1)a^{\pm}\alpha a_i^{2k-2}\right].
	\end{aligned}
\]
Hence, we conclude that the first relationship of \eqref{rel-lemma2} holds.

Analogously, by using \eqref{defsums} and \eqref{eq:dercoefci},
we compute $s^{\pm}_4$ as follows:
\begin{equation}\label{lemma2:aux1}
	\begin{aligned}
		s^{\pm}_4  
		& =  \dfrac{f^{\pm}_0}{a^{\pm}}\left[\sum_{j=2}^{2k-2} a_i^{j-2} j(j-1) C^{\pm}_{j-1} -\alpha\sum_{j=2}^{2k-2} a_i^{j-2} j(j-1) C^{\pm}_{j} \right] \\
		& = \dfrac{f^{\pm}_0}{a^{\pm}}\left[\sum_{j=2}^{2k-2} a_i^{j-2} j(j-1) C^{\pm}_{j-1} -\alpha s^{\pm}_3 \right] \\
		& = \dfrac{f^{\pm}_0}{a^{\pm}}\left[\sum_{j=1}^{2k-3} a_i^{j-1} (j+1)j C^{\pm}_{j} -\alpha s^{\pm}_3 \right] \\
		&  = \dfrac{f^{\pm}_0}{a^{\pm}}\left[\sum_{j=1}^{2k-2} a_i^{j-1} (j+1)j C^{\pm}_{j} -a_i^{2k-3} (2k-1)(2k-2)C^{\pm}_{2k-2}   -\alpha s^{\pm}_3 \right] \\
		& = \dfrac{f^{\pm}_0}{a^{\pm}}\left[\sum_{j=1}^{2k-2} a_i^{j-1} (j+1)j C^{\pm}_{j}   \mp\delta (2k-1)(2k-2)  a^{\pm}\alpha a_i^{2k-3}   -\alpha s^{\pm}_3 \right].
	\end{aligned}
\end{equation}
Finally, notice that
\begin{equation}\label{eqaux1:rel2}
		\sum_{j=1}^{2k-2} a_i^{j-1} (j+1)j C^{\pm}_{j}  =\sum_{j=1}^{2k-2} a_i^{j-1} (j-1)j C^{\pm}_{j} +2 \sum_{j=1}^{2k-2} a_i^{j-1}j C^{\pm}_{j} = a_i s^{\pm}_3 + 2s^{\pm}_1
\end{equation}
Hence, the second relationship of \eqref{rel-lemma2} follows by putting \eqref{lemma2:aux1} and \eqref{eqaux1:rel2} together.
\end{proof}

\section*{Acknowledgments}

DDN is partially supported by S\~{a}o Paulo Research Foundation (FAPESP) grants 2022/09633-5, 2019/10269-3, and 2018/13481-0, and by Conselho Nacional de Desenvolvimento Cient\'{i}fico e Tecnol\'{o}gico (CNPq) grant 309110/2021-1.
LAS is partially supported by Coordena\c{c}\~{a}o de Aperfei\c{c}oamento de Pessoal de N\'{i}vel Superior (CAPES) grant 001.

\bibliographystyle{abbrv}
\bibliography{references}

\begin{thebibliography}{10}

\bibitem{ANDRADE2023101397}
K.~S. Andrade, O.~M. Gomide, and D.~D. Novaes.
\newblock Bifurcation diagrams of global connections in filippov systems.
\newblock {\em Nonlinear Analysis: Hybrid Systems}, 50:101397, 2023.

\bibitem{andronov1966theory}
A.~Andronov, A.~Vitt, and S.~Kha{\v \i}kin.
\newblock {\em Theory of Oscillators, by A.A. Andronov, A.A. Vitt, and S.E.
  Khaikin}.
\newblock International series of monographs in physics, v. 4. Pergamon Press,
  1966.

\bibitem{CGP95}
B.~Coll, A.~Gasull, and R.~Prohens.
\newblock First {L}yapunov constants for non-smooth {L}i\'{e}nard differential
  equations.
\newblock In {\em Proceedings of the 2nd {C}atalan {D}ays on {A}pplied
  {M}athematics ({O}deillo, 1995)}, Collect. \'{E}tudes, pages 77--83. Presses
  Univ. Perpignan, Perpignan, 1995.

\bibitem{gassulcoll}
B.~Coll, A.~Gasull, and R.~Prohens.
\newblock Degenerate {H}opf bifurcations in discontinuous planar systems.
\newblock {\em Journal of Mathematical Analysis and Applications},
  253:671--690, 01 2001.

\bibitem{coll1999}
B.~Coll, R.~Prohens, and A.~Gasull.
\newblock The center problem for discontinuous {L}i{\'e}nard differential
  equation.
\newblock {\em International Journal of Bifurcation and Chaos}, 09:1751--1761,
  1999.

\bibitem{CRM}
A.~Colombo, M.~Jeffrey, J.~T. L{\'{a}}zaro, and J.~M. Olm.
\newblock {\em Extended Abstracts Spring 2016. Nonsmooth Dynamics}, volume~8 of
  {\em Trends in Mathematics}.
\newblock Birkh{\"a}user, Cham, 2017.

\bibitem{cruz19}
L.~P.~C. da~Cruz, D.~D. Novaes, and J.~Torregrosa.
\newblock New lower bound for the {H}ilbert number in piecewise quadratic
  differential systems.
\newblock {\em J. Differential Equations}, 266(7):4170--4203, 2019.

\bibitem{BBCK}
M.~di~Bernardo, C.~J. Budd, A.~R. Champneys, and P.~Kowalczyk.
\newblock {\em Piecewise-smooth dynamical systems}, volume 163 of {\em Applied
  Mathematical Sciences}.
\newblock Springer-Verlag London, Ltd., London, 2008.
\newblock Theory and applications.

\bibitem{bernardo}
M.~di~Bernardo, C.~J. Budd, A.~R. Champneys, P.~Kowalczyk, A.~B. Nordmark,
  G.~O. Tost, and P.~T. Piiroinen.
\newblock Bifurcations in nonsmooth dynamical systems.
\newblock {\em SIAM Review}, 50(4):629--701, 2008.

\bibitem{Filippov88}
A.~F. Filippov.
\newblock {\em Differential equations with discontinuous righthand sides},
  volume~18 of {\em Mathematics and its Applications (Soviet Series)}.
\newblock Kluwer Academic Publishers Group, Dordrecht, 1988.
\newblock Translated from the Russian.

\bibitem{GasTor03}
A.~Gasull and J.~Torregrosa.
\newblock Center-focus problem for discontinuous planar differential equations.
\newblock {\em Internat. J. Bifur. Chaos Appl. Sci. Engrg.}, 13(7):1755--1765,
  2003.

\bibitem{Glendinning2023}
P.~Glendinning, S.~John~Hogan, M.~Homer, M.~R. Jeffrey, and R.~Szalai.
\newblock Uncountably many cases of filippov’s sewed focus.
\newblock {\em Journal of Nonlinear Science}, 33(4), May 2023.

\bibitem{GouTor20}
L.~F.~S. Gouveia and J.~Torregrosa.
\newblock 24 crossing limit cycles in only one nest for piecewise cubic
  systems.
\newblock {\em Appl. Math. Lett.}, 103:106189, 6, 2020.

\bibitem{guardia2011generic}
M.~Guardia, T.~Seara, and M.~A. Teixeira.
\newblock Generic bifurcations of low codimension of planar {F}ilippov systems.
\newblock {\em Journal of Differential Equations}, 250(4):1967--2023, 2011.

\bibitem{HB}
J.~Harris and B.~Ermentrout.
\newblock Bifurcations in the {W}ilson--{C}owan equations with nonsmooth firing
  rate.
\newblock {\em SIAM Journal on Applied Dynamical Systems}, 14(1):43--72, 2015.

\bibitem{Mike18}
M.~R. Jeffrey.
\newblock The ghosts of departed quantities in switches and transitions.
\newblock {\em SIAM Review}, 60(1):116--136, 2018.

\bibitem{knuth97}
D.~E. Knuth.
\newblock {\em The Art of Computer Programming, Vol. 1: Fundamental
  Algorithms}.
\newblock Addison-Wesley, Reading, Mass., third edition, 1997.

\bibitem{krivan}
V.~Křivan.
\newblock On the {G}ause predator–prey model with a refuge: A fresh look at
  the history.
\newblock {\em Journal of Theoretical Biology}, 274(1):67--73, 2011.

\bibitem{mckean}
H.~McKean.
\newblock Nagumo's equation.
\newblock {\em Advances in Mathematics}, 4(3):209--223, 1970.

\bibitem{NS}
W.~Nicola and S.~A. Campbell.
\newblock Nonsmooth bifurcations of mean field systems of two-dimensional
  integrate and fire neurons.
\newblock {\em SIAM Journal on Applied Dynamical Systems}, 15(1):391--439,
  2016.

\bibitem{malgrange1971}
L.~Nirenberg.
\newblock A proof of the {M}algrange preparation theorem.
\newblock In {\em Proceedings of {L}iverpool {S}ingularities---{S}ymposium, {I}
  (1969/70)}, pages 97--105. Lecture Notes in Mathematics, Vol. 192, 1971.

\bibitem{novaessilva2020}
D.~D. Novaes and L.~A. Silva.
\newblock Lyapunov coefficients for monodromic tangential singularities in
  {F}ilippov vector fields.
\newblock {\em Journal of Differential Equations}, 300:565--596, 2021.

\bibitem{NS22}
D.~D. Novaes and L.~A. Silva.
\newblock On the non-existence of isochronous tangential centers in {F}ilippov
  vector fields.
\newblock {\em Proc. Amer. Math. Soc.}, 150(12):5349--5358, 2022.

\bibitem{PT19}
R.~Prohens and J.~Torregrosa.
\newblock New lower bounds for the {H}ilbert numbers using reversible centers.
\newblock {\em Nonlinearity}, 32(1):331--355, 2019.

\bibitem{romanovski2009center}
V.~Romanovski and D.~Shafer.
\newblock {\em The Center and Cyclicity Problems: A Computational Algebra
  Approach}.
\newblock Birkh{\"a}user Boston, 2009.

\bibitem{SIMPSON20182439}
D.~Simpson.
\newblock A compendium of {H}opf-like bifurcations in piecewise-smooth
  dynamical systems.
\newblock {\em Physics Letters A}, 382(35):2439--2444, 2018.

\bibitem{zou}
Y.~Zou, T.~K\"{u}pper, and W.-J. Beyn.
\newblock Generalized {H}opf bifurcation for planar {F}ilippov systems
  continuous at the origin.
\newblock {\em J. Nonlinear Sci.}, 16(2):159--177, 2006.

\end{thebibliography}

\end{document}